\documentclass{amsart}

\usepackage{amssymb}
\usepackage{hyperref}
\hypersetup{
    colorlinks,
    citecolor=blue,
    filecolor=blue,
    linkcolor=blue,
    urlcolor=blue
}
\usepackage[final]{showkeys}

\newtheorem{thm}{Theorem}[section]
\newtheorem{prop}[thm]{Proposition}
\newtheorem{lem}[thm]{Lemma}
\newtheorem{cor}[thm]{Corollary}

\theoremstyle{definition}
\newtheorem{definition}[thm]{Definition}

\theoremstyle{remark}
\newtheorem{remark}[thm]{Remark}

\numberwithin{equation}{section}

\renewcommand{\phi}{\varphi}
\newcommand{\tl}[1]{\tilde{#1}} 
\newcommand{\ol}[1]{\overline{#1}} 
\newcommand{\cl}[1]{\mathcal{#1}} 

\newcommand{\comment}[1]{}
\newcommand{\todo}[1]{}

\title[Existence of MCF Singularities with Bounded Mean Curvature]{Existence of Mean Curvature Flow Singularities with Bounded Mean Curvature}
\author{Maxwell Stolarski}
\address{School of Mathematical and Statistical Sciences, Arizona State University}
\email{maxwell.stolarski@asu.edu}
\urladdr{https://math.asu.edu/node/4100}

\begin{document}

\begin{abstract}
In \cite{V94}, Vel{\'a}zquez constructed a countable collection of mean curvature flow solutions in $\mathbb{R}^N$ in every dimension $N \ge 8$.
Each of these solutions becomes singular in finite time at which time the second fundamental form blows up.
In contrast, we confirm here that, in every dimension $N \ge 8$, a nontrivial subset of these solutions has uniformly bounded mean curvature.
\end{abstract}

\maketitle

\section{Introduction}
A smooth family of embedded hypersurfaces $ \left \{ \Sigma^{N-1}(t) \subset \mathbb{R}^N \right \}_{t \in [0, T)}$ moves by mean curvature flow if
	$$\partial_t \mathbf F =  \mathbf H_{\Sigma(t)}	$$
where $\mathbf H_{\Sigma(t)}$ denotes the mean curvature vector of the hypersurface $\Sigma(t)$ and 
$\mathbf F ( \cdot, t ) : \Sigma \to \Sigma(t) \subset \mathbb{R}^N$ is a smooth family of embeddings.
Often, mean curvature flow solutions develop singularities in finite time $T < \infty$.
When the hypersurfaces are closed, \comment{Is there a generalization of this result to noncompact surfaces?? I thought I saw it but now I can't find it anywhere...} Huisken \cite{Huisken84}
showed that the second fundamental form $A_{\Sigma(t)}$ blows up at the singularity time $T< \infty$ in the sense that
	$$\limsup_{t \nearrow T} \max_{x \in \Sigma(t)}  |A_{\Sigma(t)} (x) |= \infty$$
Naturally, one might then ask if the mean curvature necessarily blows up at a finite-time singularity.
Indeed, \cite{Mantegazza11} poses this question as open problem 2.4.10 on page 42.
This problem may equivalently be referred to as the extension problem, which asks,
``if $|H_{\Sigma(t)}|$ remains uniformly bounded up to time $T$, is it always possible to smoothly extend the flow past time $T$?"
\cite{Cooper11}, \cite{LS16}, \cite{LeS10}, \cite{LeS11}, \cite{LeS11_2}, \cite{LW19}, and \cite{XYZ11}, among others, made progress on this question.
In this article, we show that in general the mean curvature need not blow up at a finite-time singularity.
\begin{thm} \label{mainThmAbridged}
	For any dimension $N \ge 8$, there exists a smooth, properly embedded
	mean curvature flow solution
	$\left \{ \Sigma^{N-1}(t) \subset \mathbb{R}^N \right \}_{t \in [0, T)}$ such that
		$$\limsup_{t \nearrow T} \sup_{x \in \Sigma(t)}  | A_{\Sigma(t)}(x) | = \infty \quad \text{ and } \quad \sup_{t \in [0, T)} \sup_{x \in \Sigma(t)} | H_{\Sigma(t)}(x) | < \infty.$$
\end{thm}
Theorem \ref{mainThmFull} provides a more precise statement of theorem \ref{mainThmAbridged}.
Vel{\'a}zquez ~\cite{V94} constructed the mean curvature flow solutions referred to in theorem \ref{mainThmAbridged}.
These solutions possess an $O(n) \times O(n)$ symmetry that simplifies the associated analysis.
Informally, the solutions converge to the Simons cone at parabolic scales around the singularity,
and converge to a smooth minimal surface desingularizing the Simons cone at the scale associated to the blow-up rate of the second fundamental form.
\cite{GS18} previously investigated a proper subset of Vel{\'a}zquez's solutions and showed that in fact the mean curvature blows-up for this subset of solutions.

While the construction provided by ~\cite{V94} yields complete, non-compact mean curvature flow solutions,
the author expects that \textit{closed} mean curvature flow solutions with the same dynamics exist.
Subsection \ref{compactifying} outlines a proof of the construction of these compact mean curvature flow solutions.
~\cite{Stolarski19} rigorously constructed the analogous Ricci flow solutions on closed topologies.
Consequently, it is expected that there exist examples of closed mean curvature flow solutions satisfying the conclusion of theorem \ref{mainThmAbridged}.

The proof of theorem \ref{mainThmAbridged} is based on Liouville-type theorems and a blow-up argument.
We begin with an overview of $O(n) \times O(n)$-invariant hypersurfaces in $\mathbb{R}^{2n}$ and establish notation.
Section \ref{sectVelazquez'sResult} provides an overview of the results from \cite{V94} that we invoke.
Section \ref{sectJacobiOperator} analyzes the Jacobi operator $\Delta + | A|^2$ on a particular minimal surface that will form the basis of the Liouville-type theorems.
Section \ref{sectLiouville-TypeThms} contains the Liouville-type theorems for the Jacobi operator $\Delta+ |A|^2$ on this minimal surface and the Simons cone.
Blow-up arguments show the boundedness of mean curvature in the inner and parabolic regions in the following section.
Finally, we construct barriers to deduce that the mean curvature remains bounded throughout the rest of the hypersurface in section \ref{sectEstsOutsideParabRegion}.
The appendices include additional details on $O(n) \times O(n)$-invariant hypersurfaces and a list of constants for reference.
 
\subsection*{Aknowledgements}
I would like to thank Sigurd Angenent and Dan Knopf for bringing the expectations around the Vel{\'a}zquez mean curvature flow solutions to my attention.
I thank Richard Bamler for suggesting the approach of ``semilocal maximum principles" to prove theorem \ref{mainThmAbridged}.
I thank Lu Wang and Brett Kotschwar for helpful conversations.

\section{Parametrizations of $O(n) \times O(n)$-Invariant Hypersurfaces} \label{sectParams}
Consider $\mathbb{R}^{2n} = \mathbb{R}^{n} \times \mathbb{R}^n$ with points denoted by
	$$  z =  (\mathbf x, \mathbf y) \in \mathbb{R}^{n} \times \mathbb{R}^n 		\qquad \mathbf x = ( x^1, ... , x^n) 	\quad \mathbf y = ( y^1, ... , y^n)$$

Let $\Sigma^{2n-1} \subset \mathbb{R}^{2n}$ be a hypersurface.
If $\Sigma$ is $O(n) \times O(n)$-invariant, then $\Sigma$ is determined by its intersection with the first quadrant of the $x^1 y^1$-plane
	$$ \{ ( \mathbf x, \mathbf y) \in \mathbb{R}^{2n} : \quad  x^2 = ... = x^n = y^2 = ... = y^n = 0, \quad x^1 > 0,  \quad y^1 > 0\} $$
Assume in this plane that $\Sigma$ equals the graph of a function 
	$$y^1 = Q(x^1)$$
defined for all $x^1 > 0$.
$Q$ will be referred to as the \textit{profile function} of the hypersurface.
If an $O(n) \times O(n)$ invariant hypersurface $\Sigma$ has profile function $Q$, then $\Sigma$ may be parametrized by
	$$\mathbf F: \mathbb{R}^n \times \mathbb{S}^{n-1} \to \mathbb{R}^{2n}$$
	$$\mathbf F( \mathbf x , \mathbf \theta) = ( \mathbf x , Q( | \mathbf x |)  \mathbf \theta)$$

Write $r = | \mathbf x | \ge 0$ and $Q=Q(r)$.
We use $' = \frac{d}{dr}$ to denote the derivative with respect to $r$.
$\Sigma$ is smooth if $Q$ is smooth and additionally
	$$Q(0) > 0	\quad \text{ and } \quad Q^{(odd)}(0) = 0$$
where $Q^{(odd)}$ denotes any odd-order derivative of $Q$ with respect to $r$.

As computed in appendix \ref{appendixComputations},
the induced metric, second fundamental form, and mean curvature are given by
\begin{equation*}
	g_{\Sigma}
	= 
	\left( \begin{array}{ccc}
		 1 + Q'^2  	&	0	&	0	\\
		0					&	r^2 g_{\mathbb{S}^{n-1}}	&	0	\\
		0					&	0	&	Q'^{2} g_{\mathbb{S}^{n-1}}	\\
	\end{array} \right)
\end{equation*}

\begin{equation*}
	A_\Sigma
	= \frac{1}{\sqrt{1 + Q'^2}}
	\left( \begin{array}{ccc}
		Q''	&	0 	&	0 		\\
		0							&	Q' r g_{\mathbb{S}^{n-1}}	&	0 \\
		0	&	0&	-Q g_{\mathbb{S}^{n-1}}		\\
	\end{array} \right) 
\end{equation*}

$$H = g^{ij} A_{ij} = \frac{1}{ \sqrt{1 + Q'^2} } \left( \frac{   Q'' }{1 + Q'^2} + \frac{(n-1)}{r} Q' - \frac{ (n-1)}{Q} \right)$$
Moreover, the mean curvature flow of such $O(n) \times O(n)$-invariant hypersurfaces reduces to the following partial differential equation for the profile function $Q = Q(r,t)$
\begin{equation} \label{MCF1}
	\partial_t Q = \frac{Q''}{1 + Q'^2}  + \frac{(n-1)}{r} Q'- \frac{(n-1)}{Q}		\qquad \big( (r,t) \in (0, \infty) \times (0, T) \big)
\end{equation}
Regarded as a rotationally symmetric function $Q  = Q( | \mathbf{x} |, t)$ on $\mathbb{R}^n$, equation (\ref{MCF1}) above is equivalent to the following partial differential equation on $\mathbb{R}^n \times (0, T)$
\begin{equation} \label{MCF2}
	\partial_t Q =  \sqrt{ 1 + | \nabla Q|^2} \, div \left( \frac{ \nabla Q}{ \sqrt{ 1 + | \nabla Q|^2} } \right) - \frac{(n-1)}{Q}
	\qquad \big( (\mathbf x, t ) \in \mathbb{R}^n \times (0, T) \big)
\end{equation} 
Here, $div$ and $\nabla$ are taken on $n$-dimensional Euclidean space $\mathbb{E}^n$.
Note that equation (\ref{MCF2}) is equivalent to graphical mean curvature flow with an additional forcing term $-\frac{ (n-1)}{Q}$. 
In particular, if $|\nabla Q|$ is bounded then this equation is strictly parabolic.

We now note that certain distances are equivalent, which will permit us to use the distances interchangeably in later estimates.
\begin{lem} \label{equivDist}
	Let $\Sigma^{2n-1} \subset \mathbb{E}^{2n}$ be a smooth, connected $O(n) \times O(n)$-invariant hypersurface that intersects the plane $\{ \mathbf x = \mathbf 0 \} $.
	If $\Sigma_{\mathbf 0}$ denotes $\Sigma \cap \{ \mathbf x = \mathbf 0 \}$, then 
		$$\text{dist}_{\Sigma} ( ( \mathbf x, \mathbf y) , \Sigma_{\mathbf 0}) 
		\le C |\mathbf x | 
		\le  C dist_{\mathbb{E}^{2n}} ( ( \mathbf x, \mathbf y) , \Sigma_{\mathbf 0}) 
		\le C \text{dist}_{\Sigma} ( ( \mathbf x, \mathbf y) , \Sigma_{\mathbf 0}) $$
	where $C = \sqrt{ 1 + \| Q' \|_\infty^2 }$.
	
	Additionally,
	there exists a constant $C$ depending only on $Q(0)$ and $ \| Q' \|_\infty = \sup_x | Q'(x) |$ such that
	\begin{gather*}
		1 + \text{dist}_{\Sigma}( ( \mathbf x, \mathbf y) , ( \mathbf 0, \mathbf y_0) ) 
		\le C \left( 1 + | \mathbf x | \right) \\
		\le C \left( 1 + | ( \mathbf x, \mathbf y) | \right)  
		\le C^2 \left( 1 + \text{dist}_{\Sigma}( ( \mathbf x, \mathbf y) , ( \mathbf 0, \mathbf y_0) ) \right)
	\end{gather*}
	for all $(\mathbf x, \mathbf y) \in \Sigma$ and $( \mathbf 0, \mathbf y_0 ) \in \Sigma_{\mathbf 0}$.
\end{lem}
\begin{proof}
	First, note that for any $( \mathbf x, \mathbf y ) \in \Sigma$ and any $( \mathbf 0, \mathbf y' ) \in \Sigma_{\mathbf 0}$
	\begin{gather*}
		| \mathbf x | 
		\le \sqrt{ |\mathbf x |^2 + | \mathbf y - \mathbf y'|^2 }
		=| ( \mathbf x, \mathbf y) - ( \mathbf 0, \mathbf y') |\\
		\implies
		|\mathbf x| 
		\le \min_{( \mathbf 0, \mathbf y' ) \in \Sigma_{\mathbf 0}} | ( \mathbf x, \mathbf y) - ( \mathbf 0, \mathbf y') |
		= dist_{\mathbb{E}^{2n}} (  (\mathbf x, \mathbf y) , \Sigma_{\mathbf 0} )
	\end{gather*} 

	Next, for an arbitrary $( \mathbf x, \mathbf y ) \in \Sigma$, write
		$$( \mathbf x, \mathbf y ) = ( | \mathbf x | \omega_1, Q( |\mathbf x | ) \omega_2 ) 		\qquad \left( \omega_1, \omega_2 \in \mathbb{S}^{n-1} \right) $$
	and consider the path $\gamma(t) \in \Sigma$ given by
		$$\gamma(t) = \left( t | \mathbf x | \omega_1, Q( t | \mathbf x | ) \omega_2  \right) 		\qquad (0 \le t \le 1)$$
	It follows that
	\begin{equation*} \begin{aligned}
		\text{dist}_{\Sigma} ( ( \mathbf x, \mathbf y) , \Sigma_{\mathbf 0} ) 
		&\le \text{dist}_{\Sigma} ( ( \mathbf x, \mathbf y) , ( \mathbf 0, Q(0) \omega_2 ) ) \\
		&\le 
		 \int_0^1 | \dot \gamma (t) | dt	\\
		&= \int_0^1 \sqrt{ | \mathbf x |^2 + | \mathbf x |^2 Q'(t |\mathbf x | )^2 } dt \\
		&\le \sqrt{ 1 + \| Q' \|_\infty^2 } \, | \mathbf x | \\
		&\le \sqrt{ 1 + \| Q' \|_\infty^2 } \, \text{dist}_{\mathbb{E}^{2n}}\left ( ( \mathbf x , \mathbf y) , \Sigma_{\mathbf 0}\right) \\
		&\le \sqrt{ 1 + \| Q' \|_\infty^2 } \, \text{dist}_{\Sigma} ( ( \mathbf x, \mathbf y) , \Sigma_{\mathbf 0} ) 
	\end{aligned} \end{equation*}
	This completes the proof of the first part of the statement.

	It now follows that for some constant $C$ depending on $Q(0)$ and $\| Q' \|_\infty$
	\begin{equation*} \begin{aligned}
		1 + \text{dist}_{\Sigma} ( ( \mathbf x, \mathbf y), ( \mathbf 0, \mathbf y_0) )		
		&\le 1 + \text{dist}_{\Sigma} (( \mathbf x, \mathbf y), \Sigma_{\mathbf 0} ) + \pi Q(0)	\\
		&\le C( 1 + | \mathbf x | ) 	\\
		&\le C ( 1 + | (\mathbf x, \mathbf y) | )	\\
		&\le C ( 1 + | ( \mathbf x, \mathbf y ) - ( \mathbf 0, \mathbf y_0) | + | ( \mathbf 0, \mathbf y_0) | ) \\
		&\le C ( 1 + | ( \mathbf x, \mathbf y ) - ( \mathbf 0, \mathbf y_0) | + | Q(0) | ) \\
		&\le C^2 ( 1 + | ( \mathbf x, \mathbf y ) - ( \mathbf 0, \mathbf y_0) | ) \\ 
		&\le C^2 ( 1 + \text{dist}_{\Sigma} ( ( \mathbf x, \mathbf y), ( \mathbf 0, \mathbf y_0) )	 )
	\end{aligned} \end{equation*} 
	for any $( \mathbf x, \mathbf y) \in \Sigma$ and $( \mathbf 0, \mathbf y_0 ) \in \Sigma_{\mathbf 0}$.
\end{proof}

\begin{remark}
	Throughout the remainder of the article, we will use the notation ``$A \lesssim B$" to mean ``there exists a constant $C$ such that $A \le C B$"
	and ``$A \sim B$" to mean ``$A \lesssim B \lesssim A$."
	Subscripts as in ``$\lesssim_{a,b}$" indicate that the constant $C$ depends on $a$ and $b$.
	
	For sequences $\{ A_i \}_{i \in \mathbb{N}}$ and $\{ B_i \}_{i \in \mathbb{N}}$ with $A_i , B_i \ge 0$ for all $i$,
	we say
		$$A_i \ll B_i	\quad	\text{ if }	\quad	\lim_{i \to \infty} \frac{A_i}{B_i} = 0.$$
\end{remark}

\subsection{Self-Similar Solutions} \label{selfSimilarSolutions}
For added context, we include the profile functions of some $O(n) \times O(n)$-invariant self-similar mean curvature flow solutions.

\subsubsection{Simons Cone $\mathcal{C}$} \label{SimonsCone}
The Simons cone \cite{Simons68}
	$$\cl{C} = \{ ( \mathbf x , \mathbf y) \in \mathbb{R}^{n} \times \mathbb{R}^n \, : \, | \mathbf x | = | \mathbf y | \} $$ 
is a stationary mean curvature flow solution with profile function $Q(r,t) = r$.

\subsubsection{A Minimal Surface $\ol{\Sigma}$ Desingularizing the Simons Cone} \label{VMinlSurf}
For any $n \ge 4$, there exists a smooth, $O(n) \times O(n)$-invariant minimal surface $\ol{\Sigma}^{2n-1} \subset \mathbb{R}^{2n}$ that is asymptotic to the Simons cone at infinity.
The construction of this minimal surface is outlined in \cite{V94} and can also be found in the work of Alencar (see for example \cite{Alencar93} which also considers lower dimensions).
By scaling, there exists a one-parameter family of these minimal surfaces.
We let $\ol{Q}_b(r)$ denote the profile function for the surface $\ol{\Sigma}$ with $\ol{Q}(0) = b$.
These profiles are related by
	$$\ol{Q}_b(r) = b \ol{Q}_1 \left(\frac{r}{b} \right)		\qquad (b > 0)$$ 
Note that $\ol{Q}(r)$ is not given explicitly but its asymptotics are known and summarized in remark \ref{asympsQ}.
In situations where the particular minimal surface in this one-parameter family is irrelevant,
we shall often omit the ``$b$" subscript and simply write $\ol{Q}(r)$.

Proposition 2.2 in \cite{V94} shows additionally that $\ol{Q}''(r) > 0$ for all $r \ge 0$.
This result has the following corollary:
\begin{cor}		\label{posKernelElt}
	The function $u_0 : \ol{\Sigma} \to \mathbb{R}$ defined by
		$$u_0( \mathbf x, \mathbf y) = \left( \mathbf x, \mathbf y \right) \cdot \nu_{\ol{\Sigma}}$$
	is positive $u_0 > 0$ on $\ol{\Sigma}$.
\end{cor}
\begin{proof}
	By computations contained in appendix \ref{appendixComputations},
	\begin{gather*}
		u_0
		=  \langle \mathbf x, \ol{Q}( | \mathbf x |) \theta \rangle \cdot \frac{\left  \langle - \ol{Q}' \frac{ \mathbf x }{ | \mathbf x |} ,  \theta \right \rangle}{ \sqrt{ 1 + \ol{Q}'^2 }} 	
		= \left( 1 + \ol{Q}'^2 \right)^{-1/2} \left( \ol{Q} - | \mathbf x | \ol{Q}' \right)
	\end{gather*}
	In particular, $u_0 = u_0( | \mathbf x |)$ is a function of $| \mathbf x|$.
	It now suffices to show that $\ol{Q}(r) - r \ol{Q}'(r) > 0$ is positive.
	Differentiating with respect to $r$ implies
		$$\frac{d}{dr} \left( \ol{Q}(r) - r \ol{Q}'(r) \right) = - r \ol{Q}''(r) < 0		\quad \text{ for all $r > 0$,}$$
	Therefore $\ol{Q} - r \ol{Q}'$ is a decreasing function of $r$, and, moreover,
		$$\lim_{r \to + \infty} \ol{Q} (r) - r \ol{Q}'(r) = 0.$$
	because $\ol{\Sigma}$ is asymptotic to the Simons cone at $r = \infty$.
	It follows that $u_0 > 0$ for all $r \ge 0$. 
\end{proof}

\subsubsection{Shrinking Cylinder}	\label{shrinkingCylinder}
The spatially constant profile function
	$$Q(r,t) = \sqrt{ 2(n-1) (T- t) }		\qquad (r,t) \in [0, \infty) \times ( - \infty, T)$$
corresponds to a self-similarly shrinking cylinder $\mathbb{R}^{n} \times S^{n-1}$.

\subsubsection{Shrinking Sphere} \label{shrinkingSphere}
The self-similarly shrinking spheres $S^{2n-1}$ centered at the origin have profile functions 
	$$Q(r,t) = \sqrt{ 2(2n-1)(T-t) - r^2}		\qquad 0 \le r \le \sqrt{ 2(2n-1)(T-t)}$$
	
\begin{remark}
	\cite{DLN17} provides a systematic overview of closed $O(n) \times O(n)$-invariant self-shrinkers for the mean curvature flow. 
\end{remark}

\section{Vel{\'a}zquez's Result} \label{sectVelazquez'sResult}
In ~\cite{V94},
Vel{\'a}zquez proves the following result:

\begin{thm} \label{Velazquez'sResult}
(theorems 2.1 and 2.2 of \cite{V94})
 \comment{switched the $l$ in Velazquez to $k$ here}
	Let $n \ge 4$ and $k \in \mathbb{N}$ such that
		$$\lambda_k \doteqdot \frac{\alpha - 1}{2} + k > 0		\qquad \text{where } \alpha \doteqdot \frac{ -(2n-3) }{2} + \frac{1}{2} \sqrt{ (2n-1)^2 - 16(n - 1) } < 0$$
	For $T > 0$ sufficiently small, there exists a family of $O(n) \times O(n)$-invariant hypersurfaces 
	$\{ \Sigma_k^{2n-1}(t) \}_{t \in [0, T)}$		
	in $\mathbb{R}^{2n}$ which move by mean curvature flow and are such that
	\begin{enumerate}
		\item the surface's intersection with the $(x^1, y^1)$-plane,
		in the region where $|x^1| \lesssim \sqrt{T-t}$, 
		is given by the graph of a convex profile function $Q$,
		\item the parabolically rescaled hypersurfaces $(T-t)^{-\frac{1}{2}} \Sigma(t)$ converge in $C^2_{loc}$ away from the origin to the Simons cone $\mathcal{C}$,  
		\item for the constant
			$$\sigma_k \doteqdot \frac{\lambda_k}{1 + | \alpha | } > 0,$$
		the rescaled hypersurfaces $(T-t)^{- \sigma_k - \frac{1}{2} } \Sigma(t)$ converge uniformly on compact sets to one of the minimal hypersurfaces $\ol{\Sigma}$ as $t \nearrow T$, and
		\item the second fundamental form $A_{\Sigma(t)}$ blows up at a rate comparable to \\ $(T-t)^{-\sigma_k -\frac{1}{2} }$, that is
			$$\| A_{\Sigma(t)} \|_{L^\infty( \Sigma(t))} \sim (T-t)^{-\sigma_k -\frac{1}{2} }$$
	\end{enumerate}
\end{thm}
We will refer to the mean curvature flow solution $\left \{ \Sigma_k^{2n-1}(t) \right \}_{t \in [0, T)}$ as \textit{the Vel{\'a}zquez mean curvature flow solution of parameter $k$}.
The precise details of the convergence described in theorem \ref{Velazquez'sResult} above will be refined theorems \ref{innerRegSmoothCgce}, \ref{rescaleMinlSurfCgce}, and \ref{rescaleConeCgce} below.
The main result of this paper is
\begin{thm} \label{mainThmFull}
	Let $n \ge 4$ and let $k > 2$ be an even integer.
	The Vel{\'a}zquez mean curvature flow solution of parameter $k$ 
	$\left \{ \Sigma(t) = \Sigma_k^{2n-1}(t) \subset \mathbb{R}^{2n} \right \}_{t \in [0, T)}$
	has uniformly bounded mean curvature
		$$\sup_{t \in [0, T)} \sup_{z \in \Sigma(t)} | H_{\Sigma(t)} (z) | < \infty$$
\end{thm}
The proof this theorem will be completed near the end of section \ref{sectEstsOutsideParabRegion}.
The assumption that $k \in \mathbb{N}$ is even is a technical assumption included only to simplify the analysis in section \ref{sectEstsOutsideParabRegion} and it is expected that theorem \ref{mainThmFull} continues to hold without the assumption that $k$ is even.
However, the assumption $k > 2$ is necessary for theorem \ref{mainThmFull}.
Indeed, \cite{GS18} prove that the mean curvature does blow up when $k =2$, albeit at a rate slower than that of the second fundamental form.
Remark \ref{paramsThatWork} also indicates $k \ge 4$ may be a necessary restriction in the case $n = 4$.

\begin{thm} \label{innerRegSmoothCgce}
	Let $\{ \Sigma (t) \}_{t \in [0, T)}$ be the Vel{\'a}zquez mean curvature flow solution of parameter $k$.
	The rescaled hypersurfaces $\tl{\Sigma}(s)$ defined by
		$$\tl{\Sigma}(s(t)) \doteqdot (T-t)^{-\sigma_k - \frac{1}{2} } \Sigma(t)		\qquad s(t) = \frac{ 1}{2 \sigma_k} (T-t)^{-2\sigma_k}$$
	converge in $C^{\infty}_{loc}$ to $\ol{\Sigma}$ as $s \to \infty$.
\end{thm}
\begin{proof}
	Let
		$$L(p, s) = (T-t)^{-\sigma_k -\frac{1}{2}} Q \left(p (T-t)^{\sigma_k + \frac{1}{2}}, t \right)		\qquad s = \frac{1}{2 \sigma_k} (T-t)^{-2 \sigma_k}$$
	be the profile function of the rescaled hypersurfaces $\tl{\Sigma}(s)$.
	Note that if $Q$ satisfies \ref{MCF1} then $L$ solves
	\begin{equation} \label{evolEqnLp}
		\partial_s L = \frac{ \partial_{pp} L }{1 + (\partial_p L)^2} + \frac{ n-1}{p} \partial_p L - \frac{n-1}{L} - \frac{ \left( \sigma_k + \frac{1}{2} \right) }{2 \sigma_k s}  \left( p \partial_p L - L \right)
	\end{equation}
	or
	\begin{equation} \label{evolEqnLxi}
		\partial_s L = \sqrt{1 + | \nabla L |^2} \ div \left( \frac{ \nabla L}{ \sqrt{ 1 + | \nabla L |^2} } \right) - \frac{n-1}{L} - \frac{ \left( \sigma_k + \frac{1}{2} \right) }{2 \sigma_k s} \left( \xi \cdot \nabla L - L \right) 
	\end{equation}
	\comment{note that my constant on the last term is a little different from Velazquez's. it matches Guo-Sesum tho}
	if we regard $L = L ( | \xi|, s)$ as a function of $\xi \in \mathbb{R}^n$ and $\xi = x (T- t)^{- \sigma_k - \frac{1}{2}}$.
	We claim that $L \in C^{m, \beta}_{loc}$ and that, moreover, for every $R>0$ there exists an $s_0(R)$ such that
	the $C^{m,\beta}( B_R \times [s_1, s_1 + R^2] )$-norm is independent of $s_1$ if $s_1 \ge s_0$.
	
	First, by the proof of lemma 4.5 in ~\cite{V94}, there exist constants $A > 0$, $s_0$, $a > 0$, and $0 < \theta_- < 1 < \theta_+$ such that
		$$0 < \ol{Q}_a ( p/ \theta_-) \le L(p,s) \le \ol{Q}_a( p / \theta_+) 		\qquad \text{for all } (p,s) \in [0,A] \times [s_0, \infty)$$
	After possibly taking larger $s_0$ and $A$, lemma 4.2 in ~\cite{V94} implies that, for some constant $C_k$ depending only on $k$ and $\mu = \frac{C_k}{100} > 0$,
	\begin{equation*} \begin{aligned}
		 \left| L(p,s) - p - C_k p^{- |\alpha|} \right| &\le \frac{C_k}{100} p^{- |\alpha|} 		\\
		 \left| \partial_p L(p,s) - 1 + C_k | \alpha| p^{- |\alpha|-1} \right| &\le \frac{C_k}{100} p^{- | \alpha| -1}	\\
		 \left| \partial_{pp} L(p,s) - C_k |\alpha| ( | \alpha |+1) p^{- | \alpha | -2} \right| &\le \frac{C_k}{100} p^{- |\alpha| -2}		\\
	\end{aligned} \end{equation*}
	for all $A \le p \le A \sqrt{2 \sigma_k s}$ and $s \in [s_0, \infty)$.
	
	By theorem \ref{Velazquez'sResult}, $L$ is convex in $p$ and therefore locally Lipschitz in $p$ with estimate
		$$\sup_{p_1, p_2 \in [0, R]} \frac{ | L(p_1, s) - L(p_2,s) |}{|p_1 - p_2|} \le 3 \sup_{p \in [0, 3R]} |L(p,s)|$$
	for any $0 < R < \frac{1}{3A} \sqrt{2 \sigma_k s}$.
	Since $| \nabla L | = | \partial_p L|$, $L(| \xi|,s )$ is also locally Lipschitz in $\xi$.
	Moreover, the $C^0$ estimates for $L$ above imply that 
	for any $R$ there exists an $s_0=s_0(R) \sim R^2$ such that
	the Lipschitz constant is independent of $s$ for $s \ge s_0$.
	
	The Lipschitz bounds on $L$ imply that equation (\ref{evolEqnLxi}) is uniformly parabolic.
	Rewriting this equation (\ref{evolEqnLxi}) as 
		$$\partial_s L = \Delta L - \frac{ \nabla^i L \nabla^j L \nabla_i \nabla_j L}{1 + | \nabla L |^2} - \frac{ n-1}{L} - \frac{\left( \sigma_k + \frac{1}{2} \right)}{2 \sigma_k s} ( \xi \cdot \nabla L - L )$$
	interior estimates for quasilinear equations (namely theorem 1.1 in chapter VI of \cite{LSU88}) 
	now implies that $L$ is $C^{1, \beta}_{loc}$.
	Schauder estimates then yield that $L$ is in $C^{m, \beta}_{loc}$ for any $m$.
	Moreover, the Lipschitz bounds and coefficients in equation (\ref{evolEqnLxi}) can be bounded 
	by constants independent of $s$ when $s$ is sufficiently large depending on $R$. 
	It follows that for any $R>0$ there exists $s_0(R)$ such that the $C^{m, \beta} ( B_R \times [s_1, s_1 + R^2] )$-norm of $L$ is independent of $s_1$ for $s_1 \ge s_0$.

	
	Finally, we claim that $L$ converges in $C^{\infty}_{loc}$ to $\ol{Q}$ as $s \to \infty$.
	Suppose not for the sake of contradiction.
	Then there exists $R > 0$, $m \in \mathbb{N}$, $\epsilon >0$, and sequence of times $s_k \nearrow \infty$ such that
		$$\| L( \cdot , \cdot + s_k ) - \ol{Q} \|_{C^m( B_R \times [s_k , s_k + R^2] ) } > \epsilon$$
	By the arguments above, 
		$$\| L_k \doteqdot L( \cdot, \cdot + s_k ) \|_{C^{m, \beta} ( B_R \times [ 0 , R^2] )}$$
	is bounded by a constant independent of $k$.
	Hence, after passing to a subsequence, we may extract a limit
		$$L_k \xrightarrow[ k \to \infty ]{C^{m}( B_R \times [0, R^2] )} L_\infty$$	
	In particular, $L_k$ converges uniformly to $L_\infty$ and thus $L_\infty = \ol{Q}$.
	This however contradicts the choice of $R,m, (s_k)_{k \in \mathbb{N}}$.
\end{proof}

\begin{cor} \label{rescaleMinlSurfCgce}
	Let $\{ \Sigma(t) \}_{t \in [0, T)}$ be the Vel{\'a}zquez mean curvature flow solution of parameter $k$.
	For any sequence $t_i \nearrow T$ and $\Lambda_i \doteqdot ( T - t_i)^{- \sigma_k - \frac{1}{2}}$,
	the sequence of mean curvature flows
		$$\tl{\Sigma}_i(\tau) \doteqdot \Lambda_i \Sigma \left( t_i + \frac{ \tau}{\Lambda_i^2} \right)
		\qquad \left( \tau \in \left[- t_i \Lambda_i^2, ( T - t_i) \Lambda_i^2 \right ) \right)$$
	converges in $C^\infty_{loc}$ to $\ol{\Sigma}$ as $i \to \infty$.
\end{cor}
\begin{proof}
	The profile function $\tl{Q}_i(\xi, \tau)$ for $\tl{\Sigma}_i(\tau)$ is related to that of $\Sigma(t)$ by
	\begin{equation*} \begin{aligned} 
		\tl{Q}_i (\xi, \tau) 
		&= \Lambda_i Q \left( \xi/ \Lambda_i, t_i + \frac{ \tau}{\Lambda_i^2} \right) \\
		&= \frac{ \Lambda_i}{\Lambda(t)} 
		Q \left( \frac{ \Lambda_i}{\Lambda(t)} \frac{ \xi}{ \Lambda(t)} , t \right)		&& (t = t_i + \tau/ \Lambda_i^2) \\
		&= \frac{ \Lambda_i}{\Lambda(t_i + \tau/ \Lambda_i^2)} L \left( \frac{ \Lambda_i}{\Lambda(t_i + \tau/\Lambda_i^2)}  \xi, s(t_i + \tau/ \Lambda_i^2) \right)
	\end{aligned} \end{equation*}
	where $L$ is the profile function from the proof of theorem \ref{innerRegSmoothCgce}.
	Observe that, because $\Lambda_i^2 \gg (T-t_i)^{-1}$,
	\begin{gather*}
		\frac{ \Lambda_i}{ \Lambda(t_i + \tau/ \Lambda_i^2)}
		= \left( 1 - \frac{ \tau}{ \Lambda_i^2 (T - t_i) } \right)^{\sigma+1/2}
		\to 1
	\end{gather*}
	as $i \to \infty$ uniformly on compact $\tau$-intervals.
	Additionally,
	\begin{equation*} \begin{aligned} 
		&2 \sigma_k \left( s(t_i+ \tau_1/\Lambda_i^2) - s( t_i + \tau_0/ \Lambda_i^2 ) \right)\\
		= &( T - t_i)^{-2 \sigma} \left[ \left( 1 - \frac{ \tau_1}{ \Lambda_i^2 (T - t_i) } \right)^{- 2 \sigma} - \left( 1 - \frac{ \tau_0}{ \Lambda_i^2 (T - t_i) } \right)^{- 2 \sigma} \right]
	\end{aligned} \end{equation*}
	Using Taylor's theorem on $x \mapsto  (1 - x)^{- 2 \sigma_k}$ to estimate the terms in brackets,
	it follows that\footnote{Here, $O \left( \frac{\tau_0}{ \Lambda_i^2 ( T - t_i)} \right)$ denotes a quantity whose absolute value is bounded by $C \frac{ \tau_0}{ \Lambda_i^2 ( T - t_i)}$ for all $i$ sufficiently large locally uniformly in $\tau_0, \tau_1$. A similar definition applies for $O \left( \frac{ ( \tau_1 - \tau_0)^2}{ \Lambda_i^4( T - t_i)^2} \right)$ and $O \left( \frac{ ( \tau_1 - \tau_0)^2}{ \Lambda_i^2( T - t_i)} \right)$.}
	\begin{equation*} \begin{aligned}
		& \qquad 2 \sigma_k \left( s(t_i+ \tau_1/\Lambda_i^2) - s( t_i + \tau_0/ \Lambda_i^2 ) \right)	\\
		&= (T- t_i)^{-2 \sigma} \left[
		2 \sigma_k \frac{ \tau_1 - \tau_0}{ \Lambda_i^2 ( T- t_i)} 
		+ O \left( \frac{ \tau_0}{\Lambda_i^2(T - t_i)} \right) \frac{ \tau_1 - \tau_0}{\Lambda_i^2 ( T- t_i)}
		+ O \left( \frac{ (\tau_1 - \tau_0)^2}{\Lambda_i^4(T - t_i)^2} \right) \right]	\\
		&=  2 \sigma_k (\tau_1 - \tau_0) 
		+ O\left( \frac{ \tau_0}{ \Lambda_i^2 ( T - t_i) } \right) ( \tau_1 - \tau_0)
		 + O\left( \frac{ (\tau_1 - \tau_0)^2}{ \Lambda_i^2 ( T - t_i) } \right)
	\end{aligned} \end{equation*}
	In particular, the $s$ and $\tau$ variables are locally uniformly Lipschitz equivalent as $i \to \infty$.
	The $C^\infty_{loc}$-convergence of $\tl{Q}_i$ to $\ol{Q}$ as $i \to \infty$ now follows from theorem \ref{innerRegSmoothCgce}.
\end{proof}

\begin{thm} \label{rescaleConeCgce}
	Let $\{ \Sigma(t) \}_{t \in [0, T)}$ be the Vel{\'a}zquez mean curvature flow solution of parameter $k$.
	For any sequence $t_i \nearrow T$ and 
	$\Lambda_i$ with 
		$$(T - t_i)^{-\frac{1}{2}} \ll \Lambda_i \ll ( T- t_i)^{- \sigma_k - \frac{1}{2}}$$
	the sequence of mean curvature flows
		$$\tl{\Sigma}_i (\tau) \doteqdot \Lambda_i \Sigma \left( t_i + \frac{\tau}{\Lambda_i^2} \right)
		\qquad \left( \tau \in \left[- t_i \Lambda_i^2, ( T - t_i) \Lambda_i^2 \right ) \right)$$
	$C^\infty_{loc}( ( \mathbb{R}^{2n} \setminus \{ 0 \} ) \times \mathbb{R} )$-converges to the Simons cone $\cl{C}$ as $i \to \infty$.
\end{thm}
\begin{proof}
	First, consider the profile function
		$$q(\rho, s) \doteqdot (T - t)^{-1/2} Q( \rho ( T - t)^{1/2} , t) 		\qquad s(t) =  - \log(T - t)$$
	for the parabolically rescaled hypersurfaces $(T - t)^{-1/2} \Sigma(t)$.
	Note that if $Q$ solves equation (\ref{MCF1}) then $q$ solves
		$$\partial_s q = \frac{ \partial_{\rho \rho} q}{1 + ( \partial_\rho q)^2} + \frac{n-1}{\rho} \partial_\rho q- \frac{n-1}{q} - \frac{1}{2} ( \rho \partial_\rho q - q)$$
	
	In \cite{V94},
	condition (2.41) in the definition of $\mathcal{A}$ and the rescaling argument in the proof of lemma 4.2 imply that
	there exist constants $A, C > 0$ such that	
		$$| q(\rho, s) - \rho | \le \frac{C  }{ \rho^{| \alpha|} } e^{- \lambda_k s}		\qquad 
		\rho \in [ A e^{- \sigma_k s}, A]$$
	Additionally, these estimates propagate to the spatial derivatives 
	in the sense that for all $j \in \mathbb{N}$ there exist $A_j , C_j > 0$ such that
		$$\rho^{j} \left| \partial_\rho^j \left( q(\rho , s) - \rho \right) \right| \le \frac{C_j  }{ \rho^{| \alpha|} } e^{- \lambda_k s}		\qquad \rho \in [ A_j e^{- \sigma_k s}, A_j] $$
		
	Now consider
		$$\tl{\Sigma}_i (\tau) \doteqdot \Lambda_i \Sigma \left( t_i + \frac{ \tau}{\Lambda_i^2} \right)$$	
	and their profile functions
	\begin{equation*} \begin{aligned}
		\tl{Q}_i(\xi, \tau) 
		&= \Lambda_i Q( \xi/ \Lambda_i, t_i + \tau/ \Lambda_i^2 )	\\
		&= \Lambda_i \sqrt{T - t} \frac{1}{ \sqrt{ T - t} } Q \left( \frac{ \xi \sqrt{T - t} }{ \Lambda_i \sqrt{ T - t} }, t \right)	
		&& ( t = t_i + \tau/ \Lambda_i^2 ) \\
		&= \Lambda_i \sqrt{ T - t} \, q \left( \frac{ \xi }{ \Lambda_i \sqrt{ T - t} }, s( t_i + \tau / \Lambda_i^2 ) \right)
	\end{aligned} \end{equation*}
	
	Note that if $| \xi | \in [ r_0, R_0 ]$ for positive constants $0 < r_0 < R_0$ then
		$$ \frac{ | \xi |}{ \Lambda_i \sqrt{ T - t} }
		\le \frac{ R_0}{\Lambda_i \sqrt{ T - t} }
		= \frac{ R_0}{ \Lambda_i \sqrt{ T - t_i } } \left( 1 - \frac{ \tau }{ \Lambda_i^2 ( T - t_i) } \right)^{-1/2}
		\to 0$$
	as $i \to \infty$ uniformly on compact $\tau$-intervals.
	Additionally,
		$$ \frac{ | \xi |}{ \Lambda_i \sqrt{ T - t} } e^{\sigma_k s(t_i + \tau/ \Lambda_i^2)}
		\ge 
		\frac{r_0}{ \Lambda_i ( T - t_i)^{\frac{1}{2} + \sigma_k } } \left( 1 - \frac{ \tau}{ \Lambda_i^2 (T - t_i) } \right)^{- \frac{1}{2} - \sigma_k}
		\to \infty$$
	as $i \to \infty$ uniformly on compact $\tau$-intervals.
	In particular, for $i$ sufficiently large 
		$$A e^{- \sigma_k s} \le \frac{ |\xi|}{ \Lambda_i \sqrt{ T - t} } \le A$$
	
	
	Letting $\partial_{|\xi|}$ denote derivatives with respect to the radial variable $|\xi|$,
	we may therefore apply the estimates for $q(\rho, s)$ above to deduce 		
	\comment{using the $q$ estimate looks wrong but can be confirmed by checking how the $\partial_{|\xi|}^m | \xi|$ look for $m = 0, 1$ directly, there's just no clean way to write this}
	\begin{equation*} \begin{aligned} 
		&\,  \quad \left| \partial_{| \xi|}^m \tl{Q}_i( | \xi| , \tau) - \partial_{|\xi|}^m | \xi| \right|	\\
	 	&= \left| \frac{ \Lambda_i \sqrt{ T - t} }{ ( \Lambda_i \sqrt{T - t})^m} \, q\left( \frac{ | \xi|}{ \Lambda_i \sqrt{T - t} }, s \right) -  \partial_{|\xi|}^m | \xi| \right|	\\
		&\le \frac{ \Lambda_i \sqrt{ T - t} }{ ( \Lambda_i \sqrt{ T - t} )^m } \frac{ C ( \Lambda_i \sqrt{ T - t} )^{m + | \alpha|} }{| \xi|^{m + \alpha} } e^{- \lambda_k s}	\\
		&=  \Lambda_i^{1 + | \alpha|} ( T - t)^{ \frac{ | \alpha| + 1}{2} + \lambda_k}
		\frac{ C}{ | \xi|^{m + |\alpha|} } \\
		&= \left( \Lambda_i ( T - t_i)^{\frac{1}{2} + \sigma_k} \right)^{1 + | \alpha|} 
		\left( \frac{T - t}{T - t_i}		\right)^{ \frac{ | \alpha| + 1}{2} + \lambda_k} 
		\frac{C}{  | \xi|^{m + |\alpha|} } \\
		&= \left( \Lambda_i ( T - t_i)^{\frac{1}{2} + \sigma_k} \right)^{1 + | \alpha|} 
		\left(  1 - \frac{ \tau}{\Lambda_i^2 ( T - t_i) } \right)^{  \frac{ | \alpha| + 1}{2} + \lambda_k} 
		 \frac{C}{  | \xi|^{m + |\alpha|} }
	\end{aligned} \end{equation*}
	Because $( T - t_i)^{-1/2} \ll \Lambda_i \ll ( T - t_i)^{- \sigma_k - \frac{1}{2} }$,
	the middle factor goes to $1$ on compact $\tau$-intervals 
	and the first factor limits to $0$ as $i \to \infty$.
	$C^\infty_{loc} \left( (\mathbb{R}^{2n} \setminus \{ 0 \} ) \times \mathbb{R} \right)$-convergence follows.
\end{proof}

\section{The Jacobi Operator on $\ol{\Sigma}$ } \label{sectJacobiOperator}

In this section, we investigate the Jacobi operator $\Delta_{\ol{\Sigma}} + | \ol{A} |^2$ for the minimal surface $\ol{\Sigma}$ described in subsection \ref{VMinlSurf}.
We use overlines to refer to geometric tensors associated to $\ol{\Sigma}$.
For example, $\ol{A}$ is the second-fundamental form of $\ol{\Sigma}$, $\ol{\nu}$ the unit normal, $\ol{g}$ the induced metric, $\Delta_{\ol{\Sigma}}$ the Laplace-Beltrami operator, and $dV_{\ol{g}}$ the volume form.

\subsection{$L^2$ Theory}
Consider 
	$$\cl{L} \doteqdot \Delta_{\ol{\Sigma}} + \left| \ol{A} \right|^2$$
as an unbounded operator on $L^2 = L^2( \ol{\Sigma}, dV_{\ol{g}})$ with domain
	$$Dom(\cl{L}) = \{ u \in L^2 : \cl{L}u \in L^2 \} = \{ u \in L^2 : \Delta_{\ol{\Sigma}} u \in L^2 \} \subset H^1( \ol{\Sigma}, dV_{\ol{g}}) $$
	
We first recall some elementary properties of the Laplace-Beltrami operator $\Delta_{\ol{\Sigma}}$.
One may refer to \cite{RS1}, \cite{RS4}, \cite{Urakawa93}, and the references therein for additional background on the contents of this subsection.	
\begin{prop}
	$\Delta_{\ol{\Sigma}} : Dom( \cl{L}) \to L^2( dV_{\ol{g}})$ is a self-adjoint operator and
	$\Delta_{\ol{\Sigma}} \le 0$ in the sense that
		$$( u , \Delta_{\ol{\Sigma}} u )_{L^2} \le 0		\qquad \text{ for all } u \in Dom( \cl{L}) $$
\end{prop}

\begin{lem} \label{negEssSpec}
	Multiplication by $| \ol{A} |^2$ is relatively compact with respect to $\Delta_{\ol{\Sigma}} : Dom( \cl{L} ) \to L^2( dV_{\ol{g}})$.
	
	In particular, 
	$\cl{L} : Dom( \cl{L} ) \to L^2( dV_{\ol{g}} )$ is self-adjoint and 
	the essential spectrum satisfies
		$$\sigma_{ess} ( \cl{L} ) = \sigma_{ess} ( \Delta_{\ol{\Sigma}}) \subset ( -\infty, 0]$$
\end{lem}
\begin{proof}
	Let $f : [0, \infty) \to [0,1]$ be a smooth, nonincreasing bump function which is identically $1$ on $[0,1]$ and supported in $[0,2]$.
	Define a sequence of functions 
		$$ \eta_n : \ol{\Sigma} \to \mathbb{R}$$
		$$ \eta_n ( \mathbf x , \mathbf y) = f \left( \frac{ | \mathbf x | }{ n } \right)$$
	Note that
		$$\eta_n | \ol{A} |^2 \xrightarrow[n \to \infty]{L^\infty} | \ol{A} |^2$$
	since $ | \ol{A} |^2 \le C ( 1 + | \mathbf x |)^{-2}$.
	Moreover, it follows from lemma \ref{equivDist} that 
		$$\sup_{n \in \mathbb{N} } \| \eta_n \|_{\infty} + \| \nabla^{\ol{\Sigma}} \eta_n \|_\infty < \infty$$ 

	To prove the relative compactness of $| \ol{A} |^2$,
	let $u_k \in Dom( \cl{L} )$ be a sequence with 
		$$ ||| u_k |||^2 \doteqdot \|  u_k \|_2^2 + \|  \Delta_{\ol{\Sigma}} u_k \|_2^2$$
	uniformly bounded.
	Then $u_k$ and $\eta_n | \ol{A} |^2 u_k$ are uniformly bounded in $H^1( \ol{\Sigma}, dV_{\ol{g}} )$.
	By the choice of $\eta_n$, for each $n$, there exists a compact subset $\Omega_n \subset \ol{\Sigma}$ such that
	$\eta_n | \ol{A}|^2 u_k$ is supported in $\Omega_n$ for all $k$.
	Therefore, the Rellich theorem 
	implies that, for each $n \in \mathbb{N}$, there exists subsequence $k_j$ such that
	$\eta_n | \ol{A} |^2 u_{k_j}$ converges in $L^2( \ol{\Sigma}, dV_{\ol{g}} )$.
	In other words,
		$$\eta_n | \ol{A} |^2 : ( Dom( \cl{L} ) , ||| \cdot ||| ) \to ( L^2, \| \cdot \|_2 )$$
	is a compact operator for all $n$.
	
	Let $\Omega_n' \subset \ol{\Sigma}$ denote the complement of the domain on which $\eta_n \equiv1$.
	Observe 
			$$\| ( 1 - \eta_n ) | \ol{A} |^2  u \|_{L^2} 
			\le \| 1 - \eta_n \|_{L^\infty} \| \ol{A} \|_{L^\infty(\Omega_n')}^2 \| u \|_{L^2} 
			\le \| 1 - \eta_n \|_{L^\infty} \| \ol{A} \|_{L^\infty(\Omega_n')}^2 ||| u |||$$
	From the spatial decay of $|\ol{A}|^2$, 
	it then follows that the operators $\eta_n  | \ol{A} |^2$ converge in norm to $| \ol{A} |^2$.
	Therefore, $| \ol{A} |^2 : ( Dom( \cl{L} ) , ||| \cdot ||| ) \to ( L^2, \| \cdot \|_2 )$ is also compact
	or, equivalently, $| \ol{A} |^2$ is relatively compact with respect to $\Delta_{\ol{\Sigma}}$.
	
	The last statement of the lemma is a standard fact about relatively compact perturbations of self-adjoint operators (see for example Corollary 2 in Ch. XIII.4 of ~\cite{RS4}).
\end{proof}

Using the fact that the kernel element $( \mathbf x, \mathbf y) \cdot \ol{\nu}$ is nonvanishing in high dimensions, we can refine our understanding of the spectrum of $\cl{L}$. 
This next result is also mentioned in \cite{Alencar93} and \cite{ABPRS05}, and it alternatively follows from the results in \cite{FCS80}.

\begin{thm} \label{JacobiNonpos}
	If $n \ge 4$, then
	$\cl{L} \le 0$ and 
	the spectrum of $\cl{L}$ satisfies 
		$\sigma( \cl{L} ) \subset ( - \infty, 0]$.
\end{thm}
\begin{proof}
	Suppose for the sake of contradiction that
		$$ \lambda \doteqdot \sup_{u \in Dom( \cl{L} ), \| u \|_2 = 1} ( u , \cl{L} u)_{L^2}  \in ( 0, \| \ol{A} \|_\infty^2 ]$$
	By the min-max principle (see e.g. \cite{RS4}, chapter XIII, section 1)
	and lemma \ref{negEssSpec}, $\lambda$ is an eigenvalue of $\cl{L}$ of finite multiplicity
	and we let
		$$V_\lambda \doteqdot \{  u \in Dom ( \cl{L} )  | \cl{L} u = \lambda u \}$$
	denote the corresponding finite-dimensional eigenspace.
	
	By nondegeneracy of the ground state, all nonzero elements of $V_\lambda$ have a sign.
	It follows that $V_\lambda$ is one-dimensional.
	Indeed, if not, then we could find nonnegative functions $u, v \in V_{\lambda}$ such that 
		$$0 = ( u, v)_{L^2} \qquad  \text{ and } \qquad \| u \|_2 = \| v \|_2 = 1$$
	which is impossible.
	
	Therefore, say $V_\lambda$ is the span of $u \ge 0$.
	By standard elliptic regularity theory, $u$ is smooth and satisfies $\Delta_{\ol{\Sigma}} u + | \ol{A} |^2 u = \lambda u$ in the classical sense.
	Moreover, $u$ is $O(n) \times O(n)$-invariant since $V_\lambda$ is one-dimensional and $\cl{L}$ is $O(n) \times O(n)$-invariant.
	
	Let $u_0 :  \ol{\Sigma} \to \mathbb{R}$ be defined by $u_0 \doteqdot ( \mathbf x, \mathbf y) \cdot \ol{\nu} > 0$ as in corollary \ref{posKernelElt}.
	Note that $u_0 \sim | \mathbf x |^{\alpha}$ at infinity (see remark \ref{asympsu_0}) and hence is not in $L^2( dV_{\ol{g}} )$.
	However, by recognizing $u_0$ as the normal component of the infinitesimal generator of dilation 
	or by a direct computation,
		$$\Delta_{\ol{\Sigma}} u_0 + | \ol{A} |^2 u_0 = 0$$
	Now, write $u = u_0 w$ for some $O(n) \times O(n)$-invariant function $w \ge 0$.
	It follows that $w$ satisfies
		$$\Delta_{\ol{\Sigma}} w + 2 \left \langle \frac{ \nabla^{\ol{\Sigma}} v}{v} , \nabla^{\ol{\Sigma}} w \right \rangle_{g_{\ol{\Sigma}}} - \lambda w = 0$$
	Because $- \lambda < 0$, the maximum principle applies and
		$\sup_{\Omega} w  = \sup_{\partial \Omega} w$
	for any $\Omega \subset \ol{\Sigma}$.
	In particular, regarding $w$ as a function of $| \mathbf x|$
	and taking $O(n) \times O(n)$-invariant subsets $\Omega$, 
	it follows that $w$ is then a nondecreasing function of $|\mathbf x|$.
	This however contradicts that $u = u_0 w \in L^2$ and $u \not \equiv 0$.
	Therefore, we conclude that $\cl{L} \le 0$ and $\sigma(\cl{L}) \subset ( -\infty, 0]$.
\end{proof}

\subsection{$C^l_a$ Theory}

In this subsection, we introduce weighted $C^l$ norms on $\ol{\Sigma}$ to capture spaces of functions with suitable decay at infinity.

\subsubsection{Elliptic Estimates}
\begin{definition}
	For $a \in \mathbb{R}$, define the weighted norm $\| \cdot \|_{C^0_a( \ol{\Sigma} )}$ of a tensor $T$ on $\ol{\Sigma}$ by
		$$ \| T \|_{C^0_a( \ol{\Sigma} ) } \doteqdot \sup_{z \in \ol{\Sigma} } ( 1 + |z| )^a | T(z) |_{\ol{g}}$$
	and, for $0 < \beta < 1$, the semi-norm 
	\comment{not sure if I have the seminorm correct. See p. 4 of Deruelle-Kroencke or p. 12 of Chodosh-Schulze or p. 6 of Conlon-Hein}
		$$[ T ]_{C^{\beta}_a ( \ol{\Sigma} ) } 
		\doteqdot \sup_{0 <  dist_{\ol{\Sigma}}(z_1, z_2) \le 1 } 
		\min \left \{ \left( 1 + | z_1 | \right)^{a + \beta} , \left( 1 + | z_2 | \right)^{a + \beta} \right \} 
		\frac{ | T(z_1) - T(z_2) |_{\ol{g}} }{dist_{\ol{\Sigma}}( z_1 , z_2 )^\beta}$$
		
	For $l \ge 0$ and a function $u: \ol{\Sigma} \to \mathbb{R}$, define
		$$\| u \|_{C^l_a( \ol{\Sigma} )} \doteqdot \sum_{j = 0}^{\lfloor l \rfloor } \left \| \nabla^j_{\ol{\Sigma}} u \right\|_{C^0_{a+j} ( \ol{\Sigma}) }
		+ \left[ \nabla^{ \lfloor l \rfloor}_{\ol{\Sigma}} u \right]_{C^{ l - \lfloor l \rfloor}_{a + \lfloor l \rfloor}}$$
	where it is understood that the semi-norm term is omitted if $l \in \mathbb{N}$.
	Finally, define
		$$C^l_a ( \ol{\Sigma} ) \doteqdot \left \{ u : \ol{\Sigma} \to \mathbb{R} \, : \, \| u \|_{C^l_a( \ol{\Sigma} )} < \infty \right\}
		\qquad
		\text{ and } \qquad C^\infty_a \doteqdot \bigcap_{l \ge 0} C^l_a( \ol{\Sigma} )$$
\end{definition}

\comment{being sloppy about the norm $|x|$ now b/c have equivalences of a bunch of distance functions}
\begin{remark} \label{equivNorm}
	By lemma \ref{equivDist}, 
	we obtain equivalent norms if we replace $|  z|$ with
	$| \mathbf x |$, $dist_{\Sigma} (  z, \Sigma_{\mathbf 0 } )$, or $dist_{\mathbb{E}^{2n}} (  z, \Sigma_{\mathbf 0} )$.
\end{remark}

\comment{note that I use the gradient intrinsic to $\ol{\Sigma}, \ol{g}$. All the theory is qualitatively the same if we use the full gradient since $| \nabla X |^2 = | \ol{\nabla} X^t |^2 + |A|^2 |X^\perp|^2$.}

\begin{prop} \label{weightedSchauderEstElliptic}
	For any $l \ge 2$ and any $a \in \mathbb{R}$, $\Delta_{\ol{\Sigma}} + | \ol{A} |^2$ is a bounded operator $C^l_a \to C^{l-2}_{a+2}$.
\end{prop}
\begin{proof}
	It is clear from the definition of the weighted norms that $\Delta_{\ol{\Sigma}} : C^l_a \to C^{l-2}_{a+2}$ is bounded.
	The asymptotics of $| \ol{A} |^2$ and its derivatives imply that multiplication by $| \ol{A} |^2$ is a bounded operator $C^l_a \to C^{l-2}_{a+2}$.
\end{proof}

\subsubsection{Parabolic Estimates}


\begin{definition}
	Let $I \subset \mathbb{R}$.
	For $a \in \mathbb{R}$, define the weighted norm $ \| \cdot \|_{C^0_a( \ol{\Sigma} \times I )}$
	of a tensor $T : \ol{\Sigma} \times I \to (T^*)^p ( \ol{\Sigma} ) \times T^q ( \ol{\Sigma})$ by
		$$ \| T \|_{C^0_a( \ol{\Sigma} \times I ) } \doteqdot \sup_{(z,t) \in \ol{\Sigma} \times I } ( 1 + |z| )^a | T(z,t) |_{\ol{g}}$$
	and, for $0 < \beta < 1$, the semi-norm
	\begin{equation*} \begin{aligned}
	& [ T ]_{C^{\beta}_a ( \ol{\Sigma} \times I ) } \\
		\doteqdot & \sup_{0 < dist_{\ol{\Sigma}}(z_1, z_2) + |t_2 - t_1| \le 1 } 
		\min \left \{ \left( 1 + | z_1 | \right)^{a + \beta} , \left( 1 + | z_2 | \right)^{a + \beta} \right \} 
		\frac{ | T(z_1, t_1) - T(z_2, t_2) |_{\ol{g}} }{  dist_{\ol{\Sigma}}( z_1 , z_2 )^\beta + | t_2 - t_1|^{\beta/2}}
	\end{aligned} \end{equation*}

	For $l \ge 0$, define
		$$\| u \|_{C^l_a ( \ol{\Sigma} \times I )} \doteqdot \sum_{2i + j \le \lfloor l \rfloor } \left \|  \nabla_{\ol{\Sigma}}^j \partial_t^i u \right \|_{C^0_{a+ 2i + j} ( \ol{\Sigma} \times I )} 
		+ \sum_{2i + j = \lfloor l \rfloor } \left[ \nabla_{\ol{\Sigma}}^j \partial_t^i u \right]_{C^{ l - \lfloor l \rfloor}_{a + \lfloor l \rfloor} ( \ol{\Sigma} \times I )} $$
\end{definition}

\subsection{The Generalized Kernel}
In this subsection, we only consider functions $u = u( | \mathbf x |)$ of $| \mathbf x |$.
Let $r = | \mathbf x |$ and
use $'$ to denote $\frac{d}{dr}$.
By computations done in appendix \ref{appendixComputations},
the Jacobi operator $\Delta_{\ol{\Sigma}} + | \ol{A} |^2$ at the minimal surface $\ol{\Sigma}$
acting on such functions $u : \ol{\Sigma} \to \mathbb{R}$ becomes
	$$\frac{ u''}{ 1 + Q'^2} + \frac{ n-1}{ r} u' 
	+ \frac{1}{ 1 + Q'^2} \left[ \left( \frac{ Q''}{1 + Q'^2} \right)^2 + (n-1) \frac{ Q'^2}{ r^2} + \frac{ n-1}{Q^2} \right] u$$
or after multiplying by $1 + Q'^2$
	$$L u \doteqdot u'' + \frac{(n-1) ( 1 + Q'^2)}{r } u' + \left[ \left( \frac{ Q''}{1 + Q'^2} \right)^2 + (n-1) \frac{ Q'^2}{ r^2} + \frac{ n-1}{Q^2} \right] u$$
Here, and throughout this subsection, $Q = \ol{Q}$ is the profile function of the minimal surface $\ol{\Sigma}$.
In divergence form, the operator $L$ may be written as
	$$L u = \frac{ 1}{ \cl{J} } \frac{d}{dr} \left( \cl{J} u' \right) + V u$$
where
	$$\cl{J}(r) = \frac{ r^{n-1} Q^{n-1} }{ \sqrt{ 1 + Q'^2} } \quad \text{ and } \quad
	V(r) = \left[ \left( \frac{ Q''}{1 + Q'^2} \right)^2 + (n-1) \frac{ Q'^2}{ r^2} + \frac{ n-1}{Q^2} \right]$$
	
\begin{remark} \label{asympsQ}
	Near $r = \infty$, $Q(r)$ has the asymptotics
		$$Q(r) = \ol{Q}_b(r) = r + C_b r^{\alpha} + O(r^{\alpha -2})$$
	where $C_a > 0$.
	These asymptotics propagate to derivatives of $Q(r)$.
	Near $r = 0$
		$$Q(r) = \ol{Q}_b(r) = b + \frac{ n-1}{2n b} r^2 + O(r^4)$$
	It follows that 
	\begin{equation*}
	\begin{array}{lll}
		\cl{J}(r) \sim r^{2n-2},	& V(r) \sim \frac{1}{r^2}		 &\text{ at } r = \infty		\\
		\cl{J}(r) \sim r^{n-1},	& V(r) \sim 1				 &\text{ at } r = 0
	\end{array} \end{equation*}
	In particular, since $\cl{J}(r) > 0$, there exist positive constants $0 < c < C$ such that
		$$c ( 1 + r)^{n-1} r^{n-1} \le \cl{J}(r) \le C ( 1 + r)^{n-1} r^{n-1}$$
\end{remark}	
	
When $n \ge 4$, $u_0 \doteqdot  ( \mathbf x, \mathbf y) \cdot \ol{\nu}>0$ is a positive function of $r= | \mathbf x|$ that solves $Lu_0 = 0$.
We may therefore proceed to define the following factorization formula for $L$.
Define
	$$W \doteqdot \frac{ d }{dr} \left( \log u_0 \right) = \frac{ u_0' }{u_0}$$
	$$A u  \doteqdot - u' + W u 		\qquad A^* u \doteqdot \frac{1}{ \cl{J} } \frac{ d}{dr} \left( \cl{J} u \right) + W u$$
Then
	$$\frac{1}{ \cl{J} } \frac{ d }{dr} \left( \cl{J} u' \right) + V u = - A^* A u $$
Indeed,
	\begin{equation*} \begin{aligned}
		- A^* A u 
		&= A^* \left(  u' - W u \right) \\
		&= \frac{ 1 }{ \cl{J} } \frac{ d}{dr} ( \cl{J} u' ) + W u' - \frac{ 1}{ \cl{J} } \frac{ d }{dr} \left( \cl{J} W u \right) - W^2 u \\
		&= \frac{ 1 }{ \cl{J} } \frac{ d}{dr} ( \cl{J} u' ) + W u' 
		- W u' - W' u - \frac{ \cl{J}' }{\cl{J}} W u - W^2 u 		\\
		&= \frac{ 1 }{ \cl{J} } \frac{ d}{dr} ( \cl{J} u' ) - u \left( W' + \frac{ \cl{J}' }{ \cl{J} } W + W^2 \right)
	\end{aligned} \end{equation*}
and
	\begin{equation*} \begin{aligned}
		\left( W' + \frac{ \cl{J}' }{ \cl{J} } W + W^2 \right)
		&= \frac{ u_0''}{u_0} - \frac{ u_0'^2}{u_0^2} + \frac{ \cl{J}' }{ \cl{J} } \frac{ u_0' }{ u_0} + \frac{ u_0'^2}{ u_0^2} \\
		&= \frac{ u_0'' + \cl{J}' u_0' }{ u_0} \\
		&= -V(r)
	\end{aligned} \end{equation*}

\begin{remark} \label{asympsu_0}
	Near $r = \infty$,
		$$u_0 = \frac{ (1-\alpha) C_a }{\sqrt{2}} r^\alpha + O(r^{\alpha -2}) \quad \text{ and } \quad
		W = \frac{ \alpha}{r}  + O(r^{-3}) < 0$$
	Near $r = 0$, 
		$$u_0 \sim 1		\quad \text{ and } \quad W = O(r)$$
	In particular, since $u_0 > 0$, there exist positive constants $0 < c < C$ such that
		$$c (1 +r)^\alpha \le u_0(r) \le C (1+r)^{\alpha}$$
\end{remark}
	
\begin{lem} \label{kernel}
	The space of solutions to 
		$$Lu =\frac{ 1}{ \cl{J} } \frac{ d}{dr} \left( \cl{J} u \right) + W u = 0$$
	is two dimensional and spanned by solutions $u_0 =  (\mathbf x, \mathbf y) \cdot \ol{\nu}$ and $v_0$.
	Moreover, $u_0, v_0$ have the following asymptotics
	\begin{gather*}
		u_0(r) \sim
		\left\{ \begin{array}{cl} 
			1			& 	\text{ at } r = 0 \\
			r^{\alpha_+}	&	\text{ at } r = \infty \\
		\end{array} \right.
	\end{gather*}
	\begin{gather*}
		v_0(r) \sim
		\left\{ \begin{array}{cl} 
			r^{-(n-2)}			& 	\text{ at } r = 0 \\
			r^{\alpha_-}	&	\text{ at } r = \infty \\
		\end{array} \right.
	\end{gather*}
	where
		$$\alpha_{\pm} = \frac{1}{2} \left( -(2n-3) \pm \sqrt{(2n-3)^2-8(n-1)} \right) < 0$$
	These asymptotics propagate to all derivatives.
	In particular, $u_0 \in C^\infty_{|\alpha|}$.
\end{lem}	
\begin{proof}
	The asymptotics of the coefficients in the equation $Lu = 0$ imply that, near $r = 0$,
	$Lu = 0$ may be approximated by the Euler equation
		$$u'' + \frac{ n-1}{r} u' + \frac{ n-1}{r^2} u = 0 $$
	whose solutions are $$ C_0 + C_1 r^{-(n-2)}$$
	Similarly, near $r = \infty$, $Lu = 0$ may be approximated by the Euler equation
		$$ u'' + \frac{ 2(n-1)}{r} u' + \frac{ 2(n-1)}{r^2} u = 0$$
	whose solutions are $$C_+ r^{\alpha_+} + C_- r^{\alpha_-}$$
	
	By a standard fixed point theorem approach for perturbations of ordinary differential equations,
	it follows that there exists a basis $\tl{u}_0, \tl{v}_0$ of solutions to $Lu = 0$
	with the desired asymptotics to first order at $r = 0$.
	Since $u_0 \in C^0 ( \ol{\Sigma})$, it must be the case that $u_0 = C \tl{u}_0$.
	
	Similarly, there exists a basis of solutions $\hat{u}_0, \hat{v}_0$ with the desired asymptotics to first order at $r = \infty$.
	Since $u_0 \sim r^{\alpha_+}$ at $r = \infty$, we can perform a change of basis so that $\hat{u}_0 = u_0$.
	Thus, $\tl{v}_0 = C_1 u_0 + C_2 \hat{v}_0$.
	Taking $v_0 = \tl{v}_0 - C_1 u_0$ then yields the desired solution.
	
	Since $u_0, v_0$ satisfy the desired asymptotics to first order, 
	standard regularity theory for ordinary differential equations imply the claimed asymptotics and regularity for all derivatives of $u_0$ and $v_0$ with respect to $r$.
\end{proof}
	
We also have the following adjunction formula for suitably regular functions $u(r), v(r)$
	$$\int_0^\infty (A u ) v \cl{J} dr = \int_0^\infty u (A^* v) \cl{J} dr$$
	
From successively inverting $A$ and $A^*$,
we can define an inversion formula for the operator $L$.

\begin{definition} \label{inverseDefn}
Define inverse operators
	$$ (A^*)^{-1} f(r) = \frac{ 1}{ u_0(r) \cl{J}(r) } \int_0^r f(\rho) \cl{J}(\rho) u_0(\rho) d \rho$$
	\begin{equation*}
		A^{-1} f(r) =
		\left\{ \begin{array}{cl}
			u_0(r) \int_r^\infty \frac{ f(\rho)}{u_0 ( \rho) } d \rho, 	&		\text{if } \frac{f}{u_0} \text{ is integrable on $(0, \infty)$} \\
			-u_0(r) \int_0^r \frac{ f(\rho)}{ u_0(\rho) } d \rho,		&		\text{if } \frac{f}{u_0} \text{ is not integrable on $(0, \infty)$} \\
		\end{array} \right.
	\end{equation*}
	Finally, define
		$$ L^{-1}f = - A^{-1} (A^*)^{-1} f$$
\end{definition}

\begin{definition} \label{gendKernelDefn}
	Define the generalized kernel elements $\{ u_j \}_{j \in \mathbb{N}}$ inductively by
		$$u_j = L^{-1} \left( ( 1 + Q'^2) u_{j-1}\right) 		\qquad u_0 =  (\mathbf x, \mathbf y) \cdot \ol{\nu}$$
\end{definition}


\begin{lem} \label{gendKernelAsymps}
	For any $j \in \mathbb{N}$, there exist positive constants $0 < c_j < C_j$ such that
		$$c_j r^{2j} (1 + r)^{\alpha} \le u_j \le C_j r^{2j} ( 1 + r)^{\alpha }$$
\end{lem}
\begin{proof}
	The proof proceeds by induction with the base case having been handled by lemma \ref{kernel} and remark \ref{asympsu_0}.
	
	For the inductive step, the inductive hypothesis implies that
	\begin{equation*} \begin{aligned}
		(A^*)^{-1} \left( ( 1 + Q'^2) u_{j}\right)
		&= \frac{ 1}{u_0 \cl{J} } \int_0^r  ( 1 + Q'^2) u_j \cl{J} u_0 d \rho \\
		&\lesssim \frac{ 1}{ u_0 \cl{J} } \int_0^r C_j \rho^{2j} ( 1 + \rho )^{\alpha } \cl{J} u_0 d \rho \\
		&\lesssim 
		\frac{ 1}{ ( 1 + r)^{\alpha} (1+r)^{n-1} r^{n-1} }
		\int_0^r \rho^{2j} ( 1 + \rho )^{\alpha + n-1 + \alpha} \rho^{n-1}d \rho	\\
		&\lesssim 
		\frac{ 1}{ ( 1 + r)^{\alpha} (1+r)^{n-1} r^{n-1} } 
		( 1 + r )^{\alpha + n-1 + \alpha} r^{2j + n} 	\\
		&= r^{2j+ 1} (1 + r)^{\alpha}
	\end{aligned} \end{equation*}
	An analogous argument holds to show that
		$$(A^*)^{-1} \left( ( 1 + Q'^2) u_{j}\right) \gtrsim r^{2j+1} ( 1 + r)^\alpha$$
	In particular, 
		$$\frac{1}{u_0} (A^*)^{-1} \left( ( 1 + Q'^2) u_{j}\right) \gtrsim r^{2j+1} $$
	is not integrable on $(0, \infty)$.
	
	It follows that
	\begin{equation*} \begin{aligned}
		L^{-1} \left( ( 1 + Q'^2) u_{j}\right)
		& = u_0 \int_0^r \frac{1}{u_0} (A^*)^{-1} \left( ( 1 + Q'^2) u_{j}\right) d \rho 	\\
		& \lesssim u_0 \int_0^r \frac{ \rho^{2j + 1} ( 1 + \rho)^\alpha}{u_0} d \rho\\
		& \lesssim (1 + r)^{\alpha} \int_0^r \rho^{2j + 1} d \rho		\\
		& \lesssim r^{2j + 2} ( 1 + r)^{\alpha}
	\end{aligned} \end{equation*}
	An analogous argument holds for the lower bound and thereby completes the inductive argument.
\end{proof}

\begin{lem} \label{gendKernel}
	If $u = u (| \mathbf x |)$ is a smooth function on $\ol{\Sigma}$ that satisfies 
		$$( \Delta_{\ol{\Sigma}} + | \ol{A} |^2 )^{j} u = 0$$
	for some $j \in \mathbb{N}$,
	then
		$$u = b_0 u_ 0 + ... + b_{j-1} u_{j-1}$$
	for some constants $b_0, ... , b_{j-1}$.
\end{lem}
\begin{proof}
	We proceed by induction on $j$.
	For the base case $j = 1$, 
	first note that $u = u(r)$ and $\Delta_{\ol{\Sigma}} u + | \ol{A} |^2 u = 0$ implies
		$$L u = ( 1 + Q'^2) \left( \Delta_{\ol{\Sigma}} u + | \ol{A} |^2 u \right) =  0$$
	Hence, lemma \ref{kernel} implies that $u = b_0 u_0 + \tl{b}_0 v_0$.
	The fact that $u$ is a smooth function on $\ol{\Sigma}$ then yields $\tl{b}_0 = 0$.
	
	For the inductive step, note that
	\begin{gather*}
		0 = ( \Delta_{\ol{\Sigma}} + | \ol{A} |^2 )^{j+1} u 
		= ( \Delta_{\ol{\Sigma}} + | \ol{A} |^2 )^{j} ( \Delta_{\ol{\Sigma}} + | \ol{A} |^2 ) u 
	\end{gather*}
	so the inductive hypothesis implies 
	\begin{gather*}
		( \Delta_{\ol{\Sigma}} + | \ol{A} |^2 ) u  = b_0 u_0 + ... + b_{j-1} u_{j-1}
	\end{gather*}
	for some constants $b_0, ... , b_{j-1}$.
	Equivalently,
		$$( 1 + Q'^2)^{-1} L u = b_0 u_0 + ... + b_{j-1} u_{j-1}$$
	By lemma \ref{kernel}, it follows that for some constants $C, \tl{C}$
		$$u = C u_0 + \tl{C} v_0 + L^{-1} ( 1 + Q'^2) \left[  b_0 u_0 + ... + b_{j-1} u_{j-1} \right] 
		= C u_0 + \tl{C} v_0 + b_0 u_1 + ... + b_{j-1} u_{j}  $$
	Smoothness of $u$ on $\ol{\Sigma}$ and the asymptotics of the $u_j$ functions imply $\tl{C} = 0$.
	Equivalently, 
		$$u = C u_0  + b_0 u_1 + ... + b_{j-1} u_{j}  $$
	This completes the induction.
\end{proof}

\begin{remark}
	The interested reader is invited to compare the analysis in this subsection with that in \cite{Collot16},
	which carries out a similar analysis for a smooth stationary solution to a supercritical semilinear heat equation.
\end{remark}

\section{Liouville-Type Theorems} \label{sectLiouville-TypeThms}

In this section, we present some vanishing theorems for the minimal surfaces described in subsection \ref{selfSimilarSolutions}.

\subsection{The Minimal Surface}

\begin{lem} \label{weightedSchauderEstParab}
	For any $l > 0$ and $ a \ge 0$ there exists a constant $C$ 
	\comment{rescaling argument only works with ancient solutions...}
	such that 
	if $u \in C^l_a( \ol{\Sigma} \times (-\infty, T] )$ solves
		$$\partial_t u = \Delta_{\ol{\Sigma}} u + | \ol{A} |^2 u		\qquad (  z, t) \in \ol{\Sigma} \times (-\infty, T]$$
	then $u \in C^{l+2}_a ( \ol{\Sigma} \times (-\infty, T] )$ with
		$$\| u \|_{C^{l+2}_a ( \ol{\Sigma} \times (-\infty, T] )} \le C \| u \|_{C^l_a ( \ol{\Sigma} \times (-\infty, T])}$$
\end{lem}
\begin{proof}
	Standard local estimates apply to show that $u$ is smooth and
		$$\| u \|_{C^{l+2}_0( \ol{\Sigma} \cap B_R \times ( - \infty, T] )} \le C(n, R, l) \| u \|_{C^l_0 ( \ol{\Sigma} \cap B_R \times ( -\infty, T]) }$$
	In what remains, we shall apply a scaling argument to show the solution $u$ has the proper decay at infinity.
	Specifically, we prove the $C^{2+\beta}_a$ estimate for the associated inhomogeneous equation and 
	then repeatedly differentiate the homogeneous equation to obtain the $C^l_a$ estimate.
	
	First, let $\beta > 0$ and consider $w \in C^\beta_a ( \ol{\Sigma} \times (-\infty, T])$ a smooth solution of the inhomogeneous equation
		$$\partial_t w = \Delta_{\ol{\Sigma}} w + | \ol{A}|^2 w + f(z, t)		\qquad (z,t) \in \ol{\Sigma} \times (-\infty, T]$$
	where $f \in C^\beta_{a+2} ( \ol{\Sigma} \times (-\infty, T])$.	
	Use polar coordinates $(r, \omega)$ for $\mathbf x$ to write $w = w(r, \omega, \theta, t)$.
	By proposition \ref{RiemMetric},
	the equation
		$$\partial_t w = \Delta_{\ol{\Sigma}} w + | \ol{A}|^2 w + f$$
	can be written in coordinates as
	\begin{equation*} \begin{aligned}
		\partial_t w 
		&= \frac{ 1}{ \sqrt{ det( \ol{g}(r) ) } } \partial_i \left( \sqrt{ det ( \ol{g}(r) ) } \ \ol{g}^{ij}(r) \partial_j w \right)
		+ | \ol{A} |^2(r) w	+ f(r, \omega, \theta, t) \\
		&= \frac{ \partial_{rr} w }{ 1 + \ol{Q}_r^2 } + \frac{ n-1}{r} \partial_r w 
		+ \frac{1}{r^2} \Delta_{\mathbb{S}^{n-1}_\omega} w + \frac{1}{\ol{Q}^2} \Delta_{\mathbb{S}^{n-1}_\theta} w
		+ | \ol{A}|^2(r) w + f(r, \omega, \theta, t)
	\end{aligned} \end{equation*}
	
	Let $R > 0$, $r_0 \ge R$, and $t_0 \le T$.
	Define the rescaled function
		$$W(\rho, \omega, \theta, \tau) \doteqdot r_0^a \, w( r_0 \rho, \omega, \theta, t_0 + \tau r_0^2 )$$
	Then, for example, 
		$$
		\partial_{\rho \rho} W( \rho, \omega, \theta, \tau) =  r_0^{a+2} \partial_{rr} w(r, \omega, \theta, t)		\qquad
		\partial_\tau W( \rho, \omega, \theta, \tau) = r_0^{a+2} \partial_{t} w(r, \omega, \theta, t)$$
	where $r = r_0 \rho$ and $t = t_0 + \tau r_0^2$,
	and it follows that $W$ solves the equation
	\begin{equation*} \begin{aligned}
		\partial_\tau W 
		=& \frac{ \partial_{\rho \rho} W}{ 1 + \ol{Q}_r( r_0 \rho)^2 } + \frac{n-1}{ \rho} \partial_\rho W
		+ \frac{1}{\rho^2} \Delta_{\mathbb{S}^{n-1}_\omega} W
		+ \frac{ r_0^2}{\ol{Q}(r_0 \rho)^2} \Delta_{\mathbb{S}^{n-1}_\theta} W		\\
		&+ r_0^2 | \ol{A} |^2(r_0 \rho) W + r_0^{a+2} f( r_0 \rho, \omega, \theta, t_0 + \tau r_0^2 )
	\end{aligned} \end{equation*}
	For $(\rho, \omega, \theta, \tau) \in \left[ \frac{1}{2} , \frac{3}{2} \right] \times \mathbb{S}^{n-1} \times \mathbb{S}^{n-1} \times \left[ - \frac{1}{4}, 0 \right]$,
	the asymptotics of $\ol{Q}$ imply that the coefficients of this equation can be bounded by constants depending only on $n, R,$ and $\beta$.
	Interior estimates for parabolic equations then imply that for some constant $C = C(n,R, \beta )$ 
	\begin{equation*} \begin{aligned}
		&\,  \| W \|_{C^{2 + \beta} ( \{ ( \rho, \omega, \theta, \tau) \in [3/4, 5/4] \times \mathbb{S}^{n-1} \times \mathbb{S}^{n-1}  \times [-1/16, 0] \} ) }	\\
		\le&  C \| W \|_{C^{\beta}( \{ ( \rho, \omega, \theta, \tau) \in [1/2, 3/2] \times \mathbb{S}^{n-1} \times \mathbb{S}^{n-1}  \times [-1/4, 0] \} ) }		\\
		& + C r_0^{a + 2} \| f( r_0 \rho, \omega, \theta, t_0 + \tau r_0^2) \|_{C^{\beta} ( \{ ( \rho, \omega, \theta, \tau) \in [1/2, 3/2] \times \mathbb{S}^{n-1} \times \mathbb{S}^{n-1}  \times [-1/4, 0] \} ) }	
	\end{aligned} \end{equation*}
	Because $r_0 \ge R$ and $t_0 \le T$ were arbitrary,
	combining this $W$ estimate with the local estimate for $w$ yields
		$$\| w \|_{C^{2+\beta}_a ( \ol{\Sigma} \times ( - \infty, T] )} \le C \left( \| w \|_{C^\beta_a ( \ol{\Sigma} \times (-\infty, T] ) } 
		+ \| f \|_{C^\beta_{a+2}( \ol{\Sigma} \times (-\infty, T])} \right)	 $$
	
	The above result proves the claim for $0 < l < 1$.
	To prove the higher derivative estimates,
	let $j = \lfloor l \rfloor$ and $\beta = l - j$.
	Differentiating the evolution equation for $u$ and applying the Gauss equation, 
	it follows that any $j$th order derivative $\nabla^{(j)}_{\ol{\Sigma}} u = \nabla^{(j)}u $ satisfies an equation of the form
		$$\partial_t \nabla^{(j)} u = \Delta_{\ol{\Sigma} } \nabla^{(j)} u  + | \ol{A} |^2 \nabla^{(j)} u 
		+ \sum_{i = 0}^j \nabla^{(i)} | \ol{A} |^2 * \nabla^{(j-i)} u$$
	Note that if $u \in C^l_a( \ol{\Sigma} \times ( - \infty, T] )$ then the asymptotics of $| \ol{A}|^2$ imply that
		$$\sum_{i=0}^j \nabla^{(i)} | \ol{A} |^2 * \nabla^{(j-i)} u \in C^{\beta}_{a+j+2} ( \ol{\Sigma} \times ( -\infty, T])$$
	with the estimate
		$$\left \| \sum_{i=0}^j \nabla^{(i)} | \ol{A} |^2 * \nabla^{(j-i)} u \right \|_{C^{\beta}_{a+j+2} ( \ol{\Sigma} \times ( -\infty, T])}
		\le C(n, l, a)  \| u \|_{C^l_{a} ( \ol{\Sigma} \times ( -\infty, T] ) }$$
	Taking $w = \nabla^{(j)} u \in C^{\beta}_{a+j} ( \ol{\Sigma} \times ( -\infty, T])$, the above result for the inhomogeneous equation implies that
		$$\| \nabla^{(j)} u \|_{C^{2+ \beta}_{a+j} ( \ol{\Sigma} \times ( -\infty, T]) } \lesssim \| \nabla^{(j)} u \|_{C^{\beta}_{a+j} ( \ol{\Sigma} \times ( -\infty, T]) } + \| u \|_{C^l_{a}  ( \ol{\Sigma} \times ( -\infty, T])}$$
	Therefore,
		$\| u \|_{C^{l+2}_{a} ( \ol{\Sigma} \times ( -\infty, T]) } \lesssim \| u \|_{C^l_a ( \ol{\Sigma} \times ( -\infty, T])}$.
\end{proof}

\comment{this next theorem is based on Prop 5.2 in Brendle-Kapouleas}
\begin{thm} \label{LiouvilleMinlSurf}
	Let $n \ge 4$ and
	$u = u( | \mathbf x |, t)$ be a smooth, ancient solution to 
		$$\partial_t u = \Delta_{\ol{\Sigma}} u + | A|^2 u		\qquad	 (  z, t) \in  \ol{\Sigma} \times (-\infty, T]$$
	If there exist constants $C_0 > 0$ and $a > | \alpha|$ such that
		$$| u(  z , t) | \le \frac{ C_0 }{ (1 + | \mathbf x |)^{a} }		
		\qquad \text{ for all } ( z , t) \in  \ol{\Sigma} \times (-\infty, T]$$
	then $u \equiv 0$.
\end{thm}
\begin{proof}
	By lemma \ref{weightedSchauderEstParab}, 
	for any $j \in \mathbb{N}$, there exists $C_j$ such that
		$$| ( \Delta_{\ol{\Sigma}} + |A|^2 )^j u | \le \frac{ C_j}{( 1 + r)^{a + 2j}}$$
	In particular, 
	there exists $j = j( \alpha, n, a)$ large enough such that
		$$v(\cdot, t) \doteqdot ( \Delta_{\ol{\Sigma}} + |A|^2 )^j u( \cdot, t) \in L^2 ( dV_{\ol{g}} )$$
	is an ancient solution to $\partial_t v = \cl{L} v$ 
	with $\| v(\cdot , t) \|_{L^2( dV_{\ol{g}} )}$ uniformly bounded by a constant independent of $t$.	
	
	It follows that
		$$\frac{d}{dt} \frac{1}{2} \int_{\ol{\Sigma}} v^2 d V_{\ol{g}} = \int_{\ol{\Sigma}} \langle \cl{L} v, v \rangle_{\ol{g}} dV_{\ol{g}} = ( \cl{L} v, v)_2 \le 0$$
	Hence, $\| v(t) \|_{L^2( \ol{\Sigma} )}$ is nonincreasing in $t$ and the limit
	$M = \lim_{t \to -\infty} \| v(t) \|_{L^2( \ol{\Sigma})}$ exists.
	
	Now, take a sequence $t_i \to -\infty$ and define
	$v_i(\cdot , t) \doteqdot v( \cdot , t + t_i)$.
	By standard parabolic estimates,
	we may pass to a subsequence, still denoted by $v_i$, that converges in $C^\infty_{loc}$ to some function $v_{-\infty}( | \mathbf x |, t)$ which satisfies
		$$\partial_t v_{-\infty} = \cl{L} v_{-\infty} 		\qquad \text{ for all } (  z , t) \in \ol{\Sigma} \times \mathbb{R}$$
	Moreover, $| v_{-\infty} ( z, t) | \le C_j ( 1 + | \mathbf x |)^{a+2j}$ and
		$$\| v_{-\infty} ( \cdot , t) \|_{L^2} = \lim_{i \to \infty} \| v ( \cdot , t+t_i ) \|_{L^2} = M$$
	for all $t \in \mathbb{R}$ by the dominated convergence theorem.
	Therefore, 
		$$0 = \frac{d}{dt} \frac{1}{2} \| v_{-\infty}( \cdot , t) \|_{L^2}^2 =  ( \cl{L} v_{-\infty} , v_{-\infty} )_{L^2} $$
	and so $\cl{L} v_{-\infty} = 0$ by theorem \ref{JacobiNonpos}.
	Lemma \ref{gendKernel} and the fact that $v_{-\infty} \in L^2$ then imply that $v_{-\infty} \equiv 0$.
	Consequently, $M = 0$ and monotonicity of the $L^2$-norm yields $v \equiv 0$.
	
	Thus, 
		$$( \Delta_{\ol{\Sigma}} + \left| \ol{A} \right|^2 )^j u \equiv 0$$
	By lemma \ref{gendKernel}, $u$ may be written as
		$$u(t) =  b_0(t) u_0 + ... + b_{j-1}(t) u_{j-1}$$
	Using the spatial asymptotics at $r = \infty$ and lemma \ref{gendKernelAsymps}, it follows that
		$$0 \equiv b_0(t) \equiv b_1(t) \equiv ... \equiv b_{j-1}(t)$$
\end{proof}

\subsection{The Minimal Cone}
By remark \ref{A^2Cone} and corollary \ref{LaplaceBeltramiMinl},
when $\Sigma$ is the Simons cone $\mathcal{C}$ given by the profile function $Q(r) = r$,
the equation
	$$\partial_t u = \Delta_{\Sigma} u + | A|^2 u$$
applied to functions of $r = | \mathbf x |$ becomes
\begin{equation} \label{JacobiParabEqnCone}
	\partial_t u = \frac{1}{2} u'' + \frac{n-1}{r} u' + \frac{ n-1}{r^2} u 		\qquad r \in (0, \infty)
\end{equation}

If $u( |\mathbf x |,t)$ solves (\ref{JacobiParabEqnCone}), then 
	$$v(| \mathbf x |, t) \doteqdot |\mathbf x |^{n-1} u( | \mathbf x |, 2t)$$
solves the Bessel parabolic equation
\begin{equation}  \label{BesselParabEqn}
	\partial_t v = \Delta_\mu v
\end{equation}
where \comment{Took $\mu_+ = \mu$ throughout} 
$$\Delta_\mu \doteqdot \partial_{rr} + \left( \frac{1}{4} - \mu^2 \right) r^{-2} \qquad
	\mu = \mu (n) \doteqdot  \sqrt{ \frac{1}{4} + (n-1)(n-4) } $$
We refer the reader to \cite{BdLC18}, \cite{BCS14}, and the references therein for additional background on the Bessel parabolic equation.
Note that 
	$$| u | \le C r^{\alpha} \qquad
\text{ if and only if } \qquad
	| v | \le C r^{n-1 + \alpha} = C r^{\mu + 1/2}$$
Moreover, the stationary solution $u = C r^{\alpha}$ of (\ref{JacobiParabEqnCone}) corresponds to the stationary solution 
$v = C r^{\mu + 1/2} = C r^{n-1+\alpha}$ of (\ref{BesselParabEqn}).
	
\begin{remark}
	Note $\mu(n)$ is increasing in $n$ and some values of $\mu(n)$ include
		$$\mu(4) = \frac{1}{2} 		\qquad \mu (5) = \sqrt{ 17/4 } > 2$$
\end{remark}	

The heat kernel associated to $\Delta_\mu$ is given by
	$$W_t^\mu (r,\rho) = \frac{ \sqrt{r \rho} }{2t} I_\mu \left( \frac{r \rho}{2t} \right) e^{ - \frac{ r^2 + \rho^2}{4t} }		\qquad r, \rho, t \in (0, \infty)$$
Here $I_\mu$ denotes the modified Bessel function of the first kind and order $\mu$.

By estimating the heat kernel, we can obtain the following vanishing theorem for solutions to (\ref{BesselParabEqn}).

\begin{thm} \label{LiouvilleBesselParabEqn}
	Let $v(r,t)$ be an ancient solution to 
		$$\partial_t v = \Delta_\mu v		\qquad (r,t) \in (0, \infty) \times ( - \infty, T]$$
	If there exist constants $C> 0$ and $\delta \in (0, 2 \mu + 2 )$ such that
		$$|v(r,t)| \le C r^{\mu + 1/2- \delta}		\qquad \text{for all } (r,t) \in (0, \infty) \times ( - \infty, T]$$
	then $v \equiv 0$.
\end{thm}
\begin{proof}
	Let $r >0$ and $t_0 \in (-\infty, T]$.
	Then, for any $t >0$,
	\begin{equation*} \begin{aligned}
		|v (r, t_0) | =& \left| \int_0^\infty W_t^\mu (r,\rho) v(\rho, t_0 - t) d \rho \right| \\
		\le& \left| \int_0^{\sqrt{t}}  W_t^\mu (r,\rho) v( \rho , t_0 - t) d \rho \right| 
		+ \left| \int_{\sqrt{t}}^\infty  W_t^\mu (r,\rho) v(\rho, t_0 - t) d\rho \right| \\	
		\doteqdot& (I) + (II)
	\end{aligned} \end{equation*}
	
	Recall that the Bessel function $I_\mu$ satisfies the estimate
		$$I_\mu(z) \le  \frac{C z^\mu e^z}{(1 + z)^{\mu + 1/2} } \le C z^\mu e^{z}$$
	It follows that if $t \gg 1$ is sufficiently large depending on $r$ and $t_0$, then
	\begin{equation*} \begin{aligned}
		(I) 
		=& \left| \int_0^{\sqrt{t} }  \frac{ \sqrt{r \rho}}{2t} I_\mu \left( \frac{r \rho}{2t} \right) e^{- \frac{r^2 + \rho^2}{4t} } v(\rho, t_0 - t) d \rho \right|\\
		\lesssim& \int_0^{\sqrt{t}} \frac{ \sqrt{r \rho} }{2t} \left( \frac{r \rho}{2t} \right)^{\mu} \rho^{\mu + 1/2 - \delta} d \rho \\ 
		=& \frac{ r^{\mu + 1/2} }{(2t)^{\mu + 1} } \int_0^{\sqrt{t} } \rho^{2 \mu + 1 - \delta} d\rho \\ 
		\lesssim& r^{\mu + 1/2} t^{- \delta/2} 
		\xrightarrow[t \nearrow \infty]{} 0
	\end{aligned} \end{equation*}
	
	Also, since $r^{\mu + 1/2}$ is a stationary solution of \ref{BesselParabEqn}, 
	\begin{equation*} \begin{aligned} 
		(II) =& \left| \int_{\sqrt{t}}^\infty W_t^\mu(r,\rho) v(\rho, t_0 - t) d\rho \right| \\
		\lesssim& \int_{\sqrt{t}}^\infty W_t^\mu(r, \rho) \rho^{\mu + 1/2} \rho^{-\delta} d \rho \\
		\le& t^{- \delta/2} \int_{\sqrt{t}}^\infty W_t^\mu(r, \rho) \rho^{\mu + 1/2} d\rho \\
		\le& t^{-\delta/2} \int_0^\infty W_t^\mu(r, \rho) \rho^{\mu + 1/2} d\rho \\
		\le& t^{-\delta /2 } r^{\mu+1/2} 
		\xrightarrow[t \nearrow \infty]{} 0
	\end{aligned} \end{equation*}
	This completes the proof.
\end{proof}

Theorem \ref{LiouvilleBesselParabEqn} yields the corresponding Liouville-type theorem for solutions to (\ref{JacobiParabEqnCone}).
\begin{cor} \label{LiouvilleCone}
	Let $u = u( | \mathbf x | , t)$ be an ancient solution to 
		$$\partial_t u = \Delta_{\Sigma} u + | A_{\Sigma}|^2 u 		\qquad (  z , t) \in \Sigma \times ( - \infty, T]$$
	where $\Sigma = \mathcal{C}$ is the Simons cone.
	If there exist constants $C> 0$ and $\delta \in (0, 2 \mu + 2)$ such that
		$$| u( |\mathbf x|, t) | \le C | \mathbf x |^{\alpha - \delta}		\qquad \text{for all } (  z, t) \in \Sigma \times ( -\infty, T]$$
	then $u \equiv 0$.
\end{cor}

\section{Boundedness of $H$ in the Inner and Parabolic Regions}
In this section, we combine the Liouville-type theorems from the previous section with a blow-up argument to argue that the mean curvature remains bounded in the inner and parabolic regions up to the singularity time.
The methods in this section parallel an approach taken in \cite{BK16} and \cite{BK17}, which the latter paper refers to as ``semilocal maximum principles."

To simplify the notation, let $\Lambda(t)$ denote the blow-up rate of the second fundamental form
	$$\Lambda(t) \doteqdot (T-t)^{- \sigma_k - \frac{1}{2}} \gg \frac{1}{\sqrt{T-t}} \gg 1$$

\begin{thm} \label{bddHInner+Parab1}
	If $\Gamma > 0$,
	$a \in ( | \alpha| , | \alpha| + 1)$, and 
	$k$ is large enough such that
		$$\lambda_k \left( 1 - \frac{ a}{ 1 + | \alpha|} \right) - \frac{1}{2} \ge 0$$
	then 
		$$ \sup_{0 \le t < T} \sup_{z \in B( A \sqrt{T-t}) \cap \Sigma(t)} ( 1 + \Lambda(t) |z|)^a | H_{\Sigma(t)} (z) | < \infty$$
\end{thm}
\begin{proof}
	Suppose otherwise for the sake of contradiction.
	Then there exists a sequence of times $T_i \nearrow T$ such that
		$$\sup_{0 \le t \le T_i} \sup_{z \in B( \Gamma \sqrt{T-t}) \cap \Sigma(t)} ( 1 + \Lambda(t) |z|)^a | H_{\Sigma(t)} (z) | \doteqdot M_i \nearrow \infty$$
	Find $t_i \in [0, T_i]$ and $z_i \in \overline{ B(\Gamma \sqrt{T-t_i} ) } \cap \Sigma(t_i)$ realizing this supremum
		$$( 1 + \Lambda_i |z_i|)^a | H_{\Sigma(t_i)} (z_i) | \doteqdot M_i		
		\qquad \left( \Lambda_i \doteqdot \Lambda(t_i) \right)$$
	There are several possibilities depending on the limiting behavior of the spacetime sequence $(z_i, t_i) \in \Sigma(t_i)$.
	
	Case 1: $|z_i| \lesssim \Lambda_i^{-1}$ for some subsequence. \\
		In this case, define a sequence $\tl{\Sigma}_i(t)$ of rescaled mean curvature flows by
			$$\tl{\Sigma}_i(t) \doteqdot \Lambda_i \Sigma \left( \frac{ t}{\Lambda_i^2} + t_i \right)$$
		We obtain the following estimate on the mean curvature of the rescaled hypersurfaces $\tl{\Sigma}_i$
		\begin{equation*} \begin{aligned}
			M_i 
			=& \sup_{0 \le t \le T_i} \sup_{ z \in B( \Gamma \sqrt{T-t}) \cap \Sigma(t)} ( 1 + \Lambda(t) |z|)^a | H_{\Sigma(t)} (z) | \\
			=& \sup_{0 \le t \le t_i} \sup_{ z \in B( \Gamma \sqrt{T-t}) \cap \Sigma(t)} ( 1 + \Lambda(t) |z|)^a | H_{\Sigma(t)} (z) | \\			
			=& \sup_{0 \le t \le t_i} \sup_{ z \in B( \Gamma \sqrt{T-t}) \cap \Sigma(t)} 
				\left( 1 + \frac{ \Lambda(t)}{\Lambda_i} \Lambda_i |z| \right)^a
				\Lambda_i \left| H_{ \Lambda_i \Sigma( t)} (z \Lambda_i) \right| \\
			&(\text{setting } \, t = t_i + \tau/\Lambda_i^2		\qquad \zeta = z \Lambda_i)\\
			=& \sup_{-t_i \Lambda_i^2 \le \tau \le 0 } 	\sup_{\zeta} 
				\left( 1 + \frac{ \Lambda(t_i + \tau/\Lambda_i^2)}{\Lambda_i} |\zeta| \right)^a
				\Lambda_i \left| H_{\tl{\Sigma}_i( \tau)} (\zeta) \right| \\
		\end{aligned} \end{equation*}
	where the supremum in $\zeta$ is taken over
		$$\zeta \in B\left( \Gamma \Lambda_i \sqrt{ T - ( t_i + \tau/\Lambda_i^2)} \right) \cap \tl{\Sigma}_i(\tau) 
		\supset B \left( \Gamma \Lambda_i \sqrt{T - t_i} \right) \cap \tl{\Sigma}_i(\tau) 		\qquad (\text{if } \tau \le 0) $$
	
	Define functions
		$$ \tl{u}_i : \tl{\Sigma}_i(t) \to \mathbb{R}$$
		$$\tl{u}_i(z,t) \doteqdot \frac{\Lambda_i}{M_i} H_{\tl{\Sigma}_i(t)}(z)$$
	$\tl{u}_i(z,t)$ satisfies
		$$\partial_t \tl{u}_i = \Delta_{\tl{\Sigma}_i(t)} \tl{u}_i + | A|^2_{\tl{\Sigma}_i(t)} \tl{u}_i		\quad \text{on } \tl{\Sigma}_i(t), t \in [-t_i \Lambda_i^2, 0]$$
	and 
		$$\left( 1 + \frac{ \Lambda(t_i + t/\Lambda_i^2)}{\Lambda_i} |z| \right)^a | \tl{u}_i(z,t)| \le 1$$
	for all
		$$t \in [-t_i \Lambda_i^2, 0]		\qquad z \in B\left( \Gamma \Lambda_i \sqrt{ T - ( t_i + t/\Lambda_i^2)} \right) \cap \tl{\Sigma_i}(t) $$		
		
	Passing to the limit as $i \to \infty$ using theorem \ref{rescaleMinlSurfCgce}, we obtain an ancient solution to 
		$$\partial_t \tl{u}_\infty = \Delta_{\tl{\Sigma}_\infty} \tl{u}_\infty + | A|^2_{\tl{\Sigma}_\infty} \tl{u}_\infty$$
	defined on the limiting minimal surface $\ol{\Sigma} = \tl{\Sigma}_\infty$.
	
	Note that for any fixed $t \le 0$
		$$\lim_{i \to \infty} \frac{ \Lambda( t_i + t/ \Lambda_i^2) }{\Lambda_i}
		= \lim_{i \to \infty} \left( 1 - \frac{t}{\Lambda_i^2 ( T- t_i)} \right)^{- \sigma_k - \frac{1}{2}}
		= 1$$
	since the type II blow-up rate satisfies $\Lambda_i \gg \sqrt{ T- t_i}$.
	Hence, it follows that the limiting function $\tl{u}_\infty$ satisfies the estimate
		$$| \tl{u}_\infty(z,t) | \le ( 1 + |z|)^{-a}		\qquad \text{for all } (z,t) \in \tl{\Sigma}_\infty \times (-\infty, 0]$$
	with equality at $(\zeta_\infty, 0 ) = \lim_{i \to \infty} ( z_i \lambda_i, 0)$,
		$$| \tl{u}_\infty(\zeta_\infty, 0) | = (1 + |\zeta_\infty|)^{-a} > 0$$
	Since $a > | \alpha|$, this contradicts theorem \ref{LiouvilleMinlSurf}.\\
	
	Case 2: $ \Lambda_i^{-1} \ll | z_i| \ll \sqrt{T - t_i}$ for some subsequence.\\
	In this case, define a sequence of mean curvature flows $\tl{\Sigma}_i(t)$ by
		$$\tl{\Sigma}_i(t) \doteqdot | z_i|^{-1} \Sigma( t |z_i|^2 + t_i)$$
		
	A similar argument as in the previous case shows that
	\begin{gather*}
		M_i 
		= \sup_\tau \sup_\zeta 
			\left( 1 + \frac{ \Lambda( t_i + \tau |z_i|^2) }{\Lambda_i} \Lambda_i |z_i| |\zeta| \right)^a
			\frac{1}{|z_i|} \left| H_{\tl{\Sigma}_i(\tau)} (\zeta) \right| \\
		\ge \sup_\tau \sup_\zeta 
			\left(  \frac{ \Lambda( t_i + \tau |z_i|^2) }{\Lambda_i} \right)^a   |\zeta|^a
			\Lambda_i^a |z_i|^{a-1} \left| H_{\tl{\Sigma}_i(\tau)} (\zeta) \right| \\
	\end{gather*}
	where the suprema are taken over
		$$- t_i |x_i|^{-2} \le \tau \le 0		\qquad \zeta \in B \left( \Gamma |z_i|^{-1} \sqrt{T - ( t_i + \tau |z_i|^2 )} \right) \cap \tl{\Sigma}_i(\tau)$$
		
	Define
		$$\tl{u}_i (z,t) : \tl{\Sigma}_i(t) \to \mathbb{R}$$
		$$\tl{u}_i(z,t) \doteqdot \frac{ \Lambda_i^a |z_i|^{a-1} }{M_i} H_{\tl{\Sigma}_i(t)} (z)$$
	$\tl{u}_i$ satisfies 
		$$\partial_t \tl{u}_i = \Delta_{\tl{\Sigma}_i(t)} \tl{u}_i + | A|^2_{\tl{\Sigma}_i(t)} \tl{u}_i$$
	and
		$$\left(  \frac{ \Lambda( t_i + t |z_i|^2) }{\Lambda_i} \right)^a   |z|^a |\tl{u}_i(z,t) | \le 1$$
	for all 
		$$t \in [ - t_i |z_i|^{-2}, 0] \qquad 
		z \in B \left( \Gamma |z_i|^{-1} \sqrt{T - ( t_i + t |z_i|^2 )} \right) \cap \tl{\Sigma}_i(t)$$
		
	Passing to the limit as $i \to \infty$ using theorem \ref{rescaleConeCgce}, we obtain an ancient solution to 
		$$\partial_t \tl{u}_\infty = \Delta_{\tl{\Sigma}_\infty} \tl{u}_\infty + | A|^2_{\tl{\Sigma}_\infty} \tl{u}_\infty$$
	defined on the limiting minimal cone $\tl{\Sigma}_\infty = \mathcal{C}$.	
		
	Note that for any fixed $t \le 0$
		$$\lim_{i \to \infty} \frac{ \Lambda( t_i + t |z_i|^2) }{\Lambda_i} 
		= \lim_{i \to \infty} \left( 1 - t\frac{ |z_i|^2}{T- t_i} \right)^{- \sigma_k - \frac{1}{2} } 
		= 1$$
	since $|z_i| \ll \sqrt{ T - t_i}$.
	Hence, it follows that the limiting function $\tl{u}_\infty$ satisfies the estimate
		$$| \tl{u}_\infty(z,t) | \le |z|^{-a}$$
	Moreover, if $\zeta_i \doteqdot z_i/ |z_i| \to \zeta_\infty \in \partial B(1)$,
	\begin{equation*} \begin{aligned}
		1 
		&= \lim_{i \to \infty} (1 + \Lambda_i | z_i| |\zeta_i|)^a \frac{|z_i|^{-1}}{M_i} \left| H_{\tl{\Sigma}_i(0)} (\zeta_i) \right| \\
		&= \lim_{i \to \infty} \frac{ (1 + \Lambda_i |z_i| |\zeta_i|)^a }{\Lambda_i^a |z_i|^a} \left| \tl{u}_i (\zeta_i, 0) \right| \\
		&= \lim_{i \to \infty} \left( \Lambda_i^{-1} | z_i|^{-1} + | \zeta_i| \right)^a \left| \tl{u}_i (\zeta_i, 0) \right| \\
		&= | \zeta_\infty|^a \left| \tl{u}_\infty (\zeta_\infty, 0) \right| 
	\end{aligned} \end{equation*}
	since $\Lambda_i^{-1} \ll |z_i|$.
	In particular, $| \zeta_\infty| \in \partial B(1)$ implies
		$$ | \tl{u}_\infty ( \zeta_\infty, 0) | = | \zeta_\infty|^{-a} > 0$$
	Since $a > | \alpha|$, this contradicts theorem \ref{LiouvilleCone}.\\
	
	Case 3: $|z_i| \sim \sqrt{T - t_i}$\\
	In this final case, we may estimate the mean curvature directly from the estimates contained in \cite{V94} 
	(see for example condition (2.41) in the definition of the set $\mathcal{A}$, lemma 4.2, or lemma 4.3)
	\begin{equation*} \begin{aligned}
		M_i 
		&= ( 1 + \Lambda_i |z_i|)^{a} | H_{\Sigma(t_i)} (z_i)|\\
		&\sim \Lambda_i^a (T-t_i)^{a/2} \frac{1}{\sqrt{ T - t_i} } ( T - t_i)^{\lambda_k} \\ 
		&= ( T - t_i)^{a  \left( - \frac{ \lambda_k}{1 + |\alpha|} - \frac{1}{2} \right) + \frac{a}{2} - \frac{1}{2} + \lambda_k}\\
		&=  ( T- t_i)^{\lambda_k \left( 1 - \frac{a}{1 + |\alpha|} \right) - \frac{1}{2} }
	\end{aligned} \end{equation*}
	By assumption, $a < |\alpha| + 1$ and $k$ is sufficiently large such that
	the exponent
		$$\lambda_k \left( 1 - \frac{a}{1 + |\alpha|} \right) - \frac{1}{2} \ge 0$$
	is nonnegative.
	Hence, $M_i \lesssim 1$, which contradicts the choice of $M_i$.\\
	
	Because a contradiction arises in all possible cases, the conclusion of the theorem follows.
\end{proof}

	Since $(1 + \Lambda(t) |z|)^a \ge 1$, the theorem immediately yields the following corollary.
\begin{cor} \label{bddHInner+Parab2}
	If $\Gamma > 0$, $a \in ( | \alpha| , | \alpha| + 1)$, and $k$ is large enough such that
		$$\lambda_k \left( 1 - \frac{ a}{ 1 + | \alpha|} \right) - \frac{1}{2} \ge 0$$
	then
		$$\sup_{0 \le t < T} \sup_{z \in B( \Gamma \sqrt{T-t}) \cap \Sigma(t)}  | H_{\Sigma(t)} (z) | < \infty$$
\end{cor}

\begin{remark} \label{paramsThatWork}
	Recall that the Vel{\'a}zquez  mean curvature flows are defined for $n \ge 4$ and $k \ge 2$.
	Let us write $a = | \alpha | + \delta$ for some $\delta \in (0, 1)$.
	When $k = 2$, 
	\begin{equation*} \begin{aligned}
		\lambda_k \left( 1 - \frac{ a}{ 1 + | \alpha| } \right) - \frac{1}{2} 
		&= \left( - \frac{| \alpha | + 1 }{2} + 2 \right) \frac{ 1 - \delta}{| \alpha | + 1} - \frac{1}{2} \\
		&= 2 \left( \frac{1- \delta}{ | \alpha | + 1} \right) - \frac{1- \delta}{2} - \frac{1}{2} \\
		&< 2 \frac{ 1 - \delta}{2} - 1 + \frac{\delta }{2} 	\\
		&= - \frac{ \delta }{2} < 0
	\end{aligned} \end{equation*}
	Hence, for the lowest admissible eigenmode, the above theorem does not apply.
	In fact, \cite{GS18} show that the mean curvature blows up when $k = 2$, albeit at a rate slower than that of the second fundamental form. 
	
	However, for any $n \ge 4$ and $a \in ( | \alpha | , | \alpha + 1)$, there exists $k_0( \alpha, a) > 2$ such that the Vel{\'a}zquez mean curvature flows with $k \ge k_0$ satisfy the assumptions of the theorem.
	
	When $n = 4$ and $k \ge 4$, there exist $a \in ( | \alpha|, | \alpha | + 1)$ for which the above theorem applies.
	When $n > 4$ and $k \ge 3$, there exist $a \in ( | \alpha|, | \alpha | + 1)$ for which the above theorem applies.
	Indeed, writing $a = | \alpha | + \delta$ with $\delta \in (0, 1)$ as above, it follows that
	\begin{gather*}
		\lambda_k \left( 1 - \frac{a}{ 1 + | \alpha| } \right) - \frac{1}{2}
		= k \left( \frac{1-\delta}{1 + | \alpha|} \right) +\frac{\delta}{2} -1
	\end{gather*}
	If $n = 4$ and $k \ge 4$, then this quantity equals
	\begin{gather*}
		\lambda_k \left( 1 - \frac{a}{ 1 + | \alpha| } \right) - \frac{1}{2}
		= k \frac{1 - \delta}{3} + \frac{\delta}{2} -1	
		\ge \frac{1}{3} + \delta \left( \frac{1}{2} - \frac{4}{3} \right)	
		> 0		\qquad \text{ if $0 < \delta \ll 1$.}
	\end{gather*}
	If $n \ge 5$ and $k \ge 3$, then
	\begin{gather*}
		\lambda_k \left( 1 - \frac{a}{ 1 + | \alpha| } \right) - \frac{1}{2}
		> k \left( \frac{1-\delta}{ (2/5)} \right) + \frac{\delta}{2} -1 
		\ge \frac{1}{5} + \delta\left( \frac{1}{2} - \frac{6}{5} \right)
		> 0		\quad \text{ if $0 < \delta \ll 1$.}
	\end{gather*}	
\end{remark}

\section{Estimates Outside the Parabolic Region} \label{sectEstsOutsideParabRegion}
It remains to bound the mean curvature in the region where $r \gtrsim \sqrt{ T - t}$.
To do so, partition this region into what we call the \textit{outer region} $\left \{  r \ge \Upsilon \sqrt{T} \right \}$
and the \textit{outer-parabolic overlap} $\left\{ \Gamma \sqrt{ T - t} \le r \le \Upsilon \sqrt{T} \right \}$ with constants $\Gamma, \Upsilon$ to be determined.
First, we establish curvature estimates in the outer region.
These estimates allow us to construct barriers in the parabolic-outer overlap that subsequently bound the mean curvature in this domain.

\begin{remark} \label{adjustInitData}
Throughout this section, we will often assume the Vel{\'a}zquez mean curvature flow solutions satisfy additional bounds at $t = 0$.
So long as these bounds are consistent with the set $\mathcal{A}[\eta_1, \eta_2, \theta]$ defined in section 2 of ~\cite{V94}, there is no loss of generality in imposing these additional bounds.
Indeed, such bounds may be achieved by refining the definition of the assignment $\alpha \mapsto \Gamma_\eta(\alpha)$ in section 3 of ~\cite{V94}.
\end{remark}

\subsection{Estimates in the Outer Region}
We begin by using the interior estimates of Ecker-Huisken ~\cite{EH91} to establish curvature bounds in the outer region.

\begin{lem} \label{outerRegEst}
	Let $n \ge 4$ and $k \ge 2$.
	Assume that the Vel{\'a}zquez mean curvature flow solution of parameter $k$ $\{ \Sigma(t) \subset \mathbb{R}^{2n} \}_{t \in [0, T)}$ additionally satisfies
		$$r \le Q(r,0) \quad \text{ and } \quad  1 \le \partial_r Q(r,0) \le M_0 < \infty		\qquad \text{for all } r \ge \frac{5}{12} \Upsilon \sqrt{T}$$
	for some $\Upsilon \gg 1$ sufficiently large depending on $n$ and $M_0$.
	Then, for some constant $C_n$ depending only on $n$,
		$$| A_{\Sigma(t)}| ( \mathbf z, t) \le \frac{ C_n M_0 }{ \sqrt{t} }		\qquad 
		\text{for all }  \mathbf z \in \Sigma( t)\cap \left \{ | \mathbf x | \ge \Upsilon \sqrt{T} \right \}, \, t \in [0, T)$$
		
	In particular,
		$$| H_{\Sigma(t)}| ( \mathbf z, t) \le \frac{ C_n M_0 }{ \sqrt{t} }		\qquad 
		\text{for all }  \mathbf z \in \Sigma( t)\cap \left \{ | \mathbf x | \ge \Upsilon \sqrt{T} \right \}, \, t \in [0, T)$$
\end{lem}
\begin{proof}
Let $\rho = \frac{ 5}{6} \Upsilon \sqrt{T}$.
Let $(\mathbf x_0 , Q(| \mathbf x_0|, 0) \theta_0) \in \Sigma(0)$ with $| \mathbf x_0 | \ge \rho$.
Take 
	$$\mathbf e = \frac{1}{ \sqrt{2} } \left( -  \frac{ \mathbf x_0}{ | \mathbf x_0|}, \theta_0 \right)$$
and consider the gradient function
	$$(\nu_\Sigma( \mathbf x, \theta, t) \cdot \mathbf e)^{-1} = \sqrt{2} \sqrt{ 1 + Q_r(| \mathbf x|, t)^2} \left( \theta \cdot \theta_0 + Q_r( | \mathbf x|, t) \frac{ \mathbf x  \cdot \mathbf x_0}{ | \mathbf x | | \mathbf x_0 |} \right)^{-1}$$
For any $( \mathbf x, Q( | \mathbf x |, 0) \theta) \in \Sigma(0) \cap B \left( ( \mathbf x_0, Q( | \mathbf x_0|, 0) \theta_0 ), \frac{1}{2} \rho \right)$,
\begin{equation*} \begin{aligned}
	\frac{1}{4} \rho^2 
	&\ge | ( \mathbf x, Q \theta) - ( \mathbf x_0, Q_0 \theta_0) |^2 \\
	&= | \mathbf x|^2 + Q^2 + | \mathbf x_0|^2 + Q_0^2 
	- 2 \mathbf x_0 \cdot \mathbf x - 2 Q Q_0 \theta \cdot \theta_0 \\
	&\ge | \mathbf x |^2 + Q^2 + | \mathbf x_0 |^2 + Q_0^2 
	- | \mathbf x |^2 - \frac{ ( \mathbf x \cdot \mathbf x_0)^2}{| \mathbf x |^2 } 
	- Q^2 - Q_0^2 ( \theta \cdot \theta_0)^2 \\
	&= | \mathbf x_0|^2 \left( 1 - \left( \frac{ \mathbf x \cdot \mathbf x_0}{ | \mathbf x | | \mathbf x_0| } \right)^2 \right)
	+ Q_0^2 \left( 1 - ( \theta \cdot \theta_0)^2 \right) \\
	&\ge | \mathbf x_0|^2 \left[  \left( 1 - \left( \frac{ \mathbf x \cdot \mathbf x_0}{ | \mathbf x | | \mathbf x_0| } \right)^2 \right) + \left( 1 - ( \theta \cdot \theta_0)^2 \right) \right] 
	&& (\text{by } Q(r,0) \ge r)\\
	&\ge \rho^2 \left[  \left( 1 - \left( \frac{ \mathbf x \cdot \mathbf x_0}{ | \mathbf x | | \mathbf x_0| } \right)^2 \right) + \left( 1 - ( \theta \cdot \theta_0)^2 \right) \right] 
\end{aligned} \end{equation*}
where $Q = Q( | \mathbf x|, 0)$ and $Q_0 = Q( | \mathbf x_0|, 0)$.
Hence,
$$ \min \left \{   \frac{ \mathbf x \cdot \mathbf x_0}{ | \mathbf x | | \mathbf x_0| } ,   \theta \cdot \theta_0 \right\} \ge \frac{\sqrt{3}}{2}$$
Therefore, in $B\left( ( \mathbf x_0, Q_0 \theta_0 ) , \frac{ 1}{2} \rho \right) \cap \Sigma(0)$, the gradient function may be estimated by
\begin{equation*} \begin{aligned}
	( \nu_\Sigma \cdot \mathbf e )^{-1} 
	&= \sqrt{2} \sqrt{ 1 + Q_r( | \mathbf x|, 0)^2} \left( \theta \cdot \theta_0 + Q_r( | \mathbf x|, 0) \frac{ \mathbf x  \cdot \mathbf x_0}{ | \mathbf x | | \mathbf x_0 |} \right)^{-1} \\
	&\le \sqrt{2} \sqrt{ 1 + M_0^2} \left( \frac{ \sqrt{3} }{2} +  \frac{ \sqrt{3}}{2} \right)^{-1} \\
	&= \sqrt{ \frac{2}{3} } \sqrt{ 1 + M_0^2 }< \infty 
\end{aligned} \end{equation*}
Interior estimates for the gradient function (theorem 2.1 of ~\cite{EH91})  then imply that
	$$( \nu_{\Sigma(t)} \cdot \mathbf e )^{-1} 
	\le 
	\left( 1 - \frac{ | ( \mathbf x, Q( |\mathbf x |, t) \theta) - ( \mathbf x_0, Q_0 \theta_0 ) |^2 + 2(2n-1) t }{ (\rho/2)^2} \right)^{-1} 
	\sqrt{ \frac{2}{3} } \sqrt{ 1 + M_0^2}$$
for all 
	$\mathbf z \in \Sigma(t) \cap B \left( ( \mathbf x_0, Q( | \mathbf x_0|, 0) \theta_0 ) , \sqrt{ (\rho/2)^2 - 2(2n-1)t } \right)$.
In particular,
	$$( \nu_{\Sigma(t)} \cdot \mathbf e )^{-1} \lesssim \sqrt{ 1+ M_0^2} $$
for all
	$\mathbf z \in \Sigma(t) \cap B \left( ( \mathbf x_0, Q( | \mathbf x_0|, 0) \theta_0 ) , \sqrt{ (\rho/3)^2 - 2(2n-1)t } \right)$.
Theorem 3.1 of ~\cite{EH91} then implies
	$$| A_{\Sigma(t)} |(\mathbf z, t) \lesssim_n \sqrt{ 1 + M_0^2}  	\left( \frac{ 1}{\sqrt{ t} } + \frac{ 1}{ \rho} \right)$$
for all 
	$\mathbf z \in \Sigma(t) \cap B \left( ( \mathbf x_0, Q( | \mathbf x_0|, 0) \theta_0 ) , \sqrt{ (\rho/4)^2 - 2(2n-1)t } \right)$.
In particular, if $\Upsilon \gg 1$ is sufficiently large depending on $n$, then
	$$| A_{\Sigma(t)}| ( \mathbf z, t) \lesssim_n \frac{ \sqrt{ 1 + M_0^2} }{ \sqrt{t}} $$
for all 
	$\mathbf z \in \Sigma(t) \cap B \left( ( \mathbf x_0, Q( | \mathbf x_0|, 0) \theta_0 ) , \frac{\rho}{5} \right)$.
Since $\mathbf z_0 \in \Sigma(0)$ with $| \mathbf x_0| \ge \rho$ was arbitrary, this curvature estimate therefore holds on 
	$$\Omega(t) \doteqdot \Sigma(t) \cap \bigcup_{\mathbf z_0 \in \Sigma(0), \ | \mathbf x_0 | \ge \rho} B( \mathbf z_0, \rho/5 )$$
	
In particular,
	$$| H_{\Sigma(t)} |(\mathbf z, t) \le \frac{ C_n \sqrt{ 1 + M_0^2} }{\sqrt{t}}		\qquad \text{for all } \mathbf z \in \Omega(t)$$
If $\Sigma(t)$ is parametrized so that 
	$$\partial_t \mathbf z = H_{\Sigma(t)} \nu_{\Sigma(t)}$$
then integrating the mean curvature estimate yields
	$$| \mathbf z(t) - \mathbf z(0) | \le 2 C_n \sqrt{1+ M_0^2} \sqrt{t} \le 2 C_n \sqrt{1 + M_0^2} \sqrt{T} <  \frac{1}{6} \Upsilon \sqrt{T} = \rho/5 $$
if $\Upsilon \gg 1$ is sufficiently large depending on $n$ and $M_0$.	
It follows that 
	$$\Sigma(t) \cap \left\{\mathbf x \in \mathbb{R}^n : \, | \mathbf x | \ge \frac{6}{5} \rho \right \} \subset \Omega(t) \qquad \text{for all } t \in [0, T)$$
Therefore,
	$$| A_{\Sigma(t)} |( \mathbf z, t) \lesssim_n \frac{ \sqrt{ 1 + M_0^2}}{\sqrt{t}} \lesssim \frac{M_0}{\sqrt{t}}$$
for all
	$( \mathbf z, t) \in \Sigma(t) \cap \left\{ | \mathbf x | \ge \Upsilon \sqrt{T} \right\}$.
\end{proof}
%

\begin{remark} \label{derivEstOuter}
	Considering the components of the metric $g_{\Sigma}$ and its evolution equation $\partial_t g =-2HA$,
	lemma \ref{outerRegEst} also implies a uniform bound on $\partial_r Q$, say
		$$| \partial_r Q (r,t) | \le M 		\qquad \text{for all } r \ge \Upsilon \sqrt{T} , \quad t \in [0, T)$$

\end{remark}

\subsection{Coarse Barriers in the Parabolic-Outer Overlap} \label{coarseBarriers}
We now begin to estimate the curvature in the parabolic-outer overlap by establishing coarse estimates for the profile function $Q$.
Henceforth, we shall restrict to the case where the eigenmode $\lambda_k$ is additionally chosen so that $k \in \mathbb{N}$ is even.
This restriction will simplify some of the estimates and barriers that follow.
Indeed, \cite{V94} shows
	$$ \left| Q \left( r, t \right) - \left(  r + (T-t)^{\lambda_k +1/2} \phi_k \left( \frac{r}{\sqrt{T-t}} \right) \right) \right|
	\le \mu (T - t)^{\lambda_k +1/2}
		\quad \text{for } r \sim \sqrt{T-t}$$
where $\mu \ll 1$ is a small constant and $\phi_k$ is an eigenfunction for the differential operator $A$ in equation (2.20) of ~\cite{V94} with associated eigenvalue $\lambda_k$.
When $k$ is even, $\phi_k( x)$ is asymptotic to $\ol{C}_k x^{2 \lambda_k +1}$ as $x$ tends to infinity, where $\ol{C}_k > 0$ is a positive constant.
In particular, when $k$ is even, for sufficiently large $\Gamma$, the profile function satisfies
	$$Q\left( \Gamma \sqrt{T-t} , t \right) \ge \Gamma \sqrt{T - t}		\qquad \text{for all } t \in [0, T)$$
By remark \ref{adjustInitData}, we may therefore assume without loss of generality that the initial data is chosen so that
	$$Q(r, 0) \ge r		\qquad \text{for all } r \ge \Gamma \sqrt{ T - t}$$
and
	$$\lim_{r \to \infty} Q(r, 0) - r = \infty$$
Note that this last assumption and the proof of lemma \ref{outerRegEst} imply that $Q(r,t) > r$ for sufficiently large $r$.
The avoidance principle then implies
	$$Q(r,t) \ge r		\qquad \text{for all } r \ge \Gamma \sqrt{T - t}, \quad t \in [0, T)$$
since $Q(r,t) = r$ is a solution to the mean curvature flow equation (\ref{MCF1}).
	
Differentiating the mean curvature flow equation \ref{MCF1} with respect to $r$, it follows that $Q_r$ satisfies
	$$\partial_t Q_r = \frac{  Q_{rrr}}{1 + Q_r^2} - \frac{ 2 Q_{rr}^2 Q_r }{( 1 + Q_r^2)^2} + \frac{ (n-1)}{r} Q_{rr} + (n-1) \left( \frac{ 1}{Q^2} - \frac{1}{r^2} \right) Q_r$$

Observe that when $Q \ge r$, the coefficient $\frac{ 1}{Q^2} - \frac{1}{r^2} \le 0$.
Hence, the maximum principle implies
	$$| Q_r(r,t) | \le \sup_{(r,t) \in \partial_P \Omega} | Q_r(r,t) |		\qquad \text{for all } (r,t) \in \Omega$$
where $\Omega = \Omega( \Gamma, \Upsilon)$ is the spactime domain
	$$\Omega \doteqdot \left \{ (r,t) \in (0, \infty) \times [0, T) \, : \,  \Gamma \sqrt{ T-t} < r < \Upsilon \sqrt{T} \right\}$$
and $\partial_P \Omega$ is the parabolic boundary of this domain
\begin{equation*} \begin{aligned}
	\partial_P \Omega =& \left \{ (r,t) : r \ge \Gamma \sqrt{T - t_0}, t = 0 \right \} \\
	& \cup \left \{ (r,t) : r = \Gamma \sqrt{ T-t} , t \ge 0 \right\}	\\
	 & \cup  \left \{ (r,t) : r = \Upsilon \sqrt{ T}  , t \ge 0 \right\}	\\
\end{aligned} \end{equation*}
By replacing $M$ in remark \ref{derivEstOuter} with a possibly larger constant, we may then assume without loss of generality that
	$$| Q_r(r,t) | \le \sup_{(r,t) \in \partial_P \Omega} | Q_r(r,t) | 
	\le M < \infty	\qquad
	\text{for all } (r,t) \in \Omega , $$
	
In particular, this derivative estimate ensures that the mean curvature flow equation (\ref{MCF2}) for $Q$ is strictly parabolic in the region $\Omega$ with uniform estimates on the ellipticity constants
	$$\frac{1}{1+M^2} \le a^{ij} \le 1$$

\subsection{Finer Estimates in the Parabolic-Outer Overlap}
Recall equation (\ref{MCF1})
	$$\partial_t Q =  \frac{Q_{rr}}{1 + Q_r^2} + \frac{(n-1)}{r} Q_r - \frac{(n-1)}{Q}$$
Define
	$$v(r,t) \doteqdot Q(r,t) - r$$
to be the perturbation of the profile function from the Simons cone.
It follows that $v(r,t)$ solves
\begin{equation} \label{parabEqnPerturbation}
\begin{aligned}
	\partial_t v &= \frac{1}{1 + Q_r^2 } v_{rr} + \frac{(n-1)}{r} v_r - \frac{(n-1)}{r} \left[ \frac{1}{1 + (v/r)} - 1 \right]		\\
	&= \frac{1}{1 + Q_r^2 } v_{rr} + \frac{(n-1)}{r} v_r + \frac{(n-1)}{r^2} v
	- \frac{(n-1)}{r}  \left[ \frac{1}{1 + (v/r)} - 1 + \frac{v}{r} \right]
\end{aligned} \end{equation}
To estimate $Q = r + v$, we shall find a positive supersolution $0 < v^+$ to this equation in the region $\Omega$.

\subsubsection{A Positive Supersolution $v^+$} \label{supersolnSubSubSect}
The search for a positive supersolution $v^+$ is aided by the fact that the function $x \mapsto \frac{1}{1 + x}$ is convex for nonnegative $x$.
Hence, 
	$$- \frac{(n-1)}{r}  \left[ \frac{1}{1 + (v/r)} - 1 - \frac{v}{r} \right] \le 0$$
for nonnegative $v$.

	
\begin{lem} \label{supersolnLem}
\comment{Oddly enough, we don't need the $Q_r$ estimate for the supersolution.}
	For any $\lambda_k = \lambda > 0$ and $C_0 > 0$, define
		$$C_1 \doteqdot \left[  ( 2 \lambda + 1) ( 2 \lambda) + (n-1)(2 \lambda + 1 ) + (n-1) \right] C_0 > 0$$
	Then 
		$$v^+(r,t) \doteqdot C_0 r^{2 \lambda + 1} - C_1 (T-t) r^{2 \lambda -1}$$
	is a supersolution to equation (\ref{parabEqnPerturbation}) on the domain where $v^+ > 0$.
	
	If $C_0 > \overline{C}_k$ and $\Gamma \gg 1$ is sufficiently large depending on $n, k,$ and $C_0$, 
	then
		$$v^+(r,t) > \overline{C}_k r^{2 \lambda +1} 		\qquad \text{for all } r = \Gamma \sqrt{T-t}, \, t \in [0, T)$$
		$$\text{ and } \qquad v^+(r,t) > 0			\qquad \text{for all } r \ge \Gamma \sqrt{T - t}, \, t \in [0, T)$$
\end{lem}
\begin{proof}
	As noted above, when $v^+ > 0$,
		$$- \frac{(n-1)}{r}  \left[ \frac{1}{1 + (v/r)} - 1 - \frac{v}{r} \right] \le 0$$
	Hence, it suffices to show that $v^+$ is a supersolution to 
		$$\partial_t v  = \frac{1}{1 + Q_r^2} v_{rr} + \frac{(n-1)}{r} v_r + \frac{(n-1)}{r^2} v$$
	A direct computation shows that
	\begin{equation*} \begin{aligned}
		&\partial_t v^+ - \left( \frac{1}{1 + Q_r^2} v^+_{rr} + \frac{(n-1)}{r} v^+_r + \frac{(n-1)}{r^2} v^+ \right) \\
		=& C_1 r^{2 \lambda - 1} \\
		&+ C_1 ( T-t) r^{2 \lambda - 3} \left( \frac{1}{1 + Q_r^2} ( 2 \lambda - 1) ( 2 \lambda - 2) + (n-1) (2 \lambda - 1) + (n-1) \right)  \\
		&- C_0 r^{ 2\lambda -1} \left( \frac{1}{1 + Q_r^2} ( 2 \lambda + 1) ( 2 \lambda ) + (n-1) ( 2 \lambda + 1) + ( n-1) \right) \\
		\ge& C_1 r^{2 \lambda - 1} 
		- C_0 r^{ 2\lambda -1} \left( ( 2 \lambda + 1) ( 2 \lambda ) + (n-1) ( 2 \lambda + 1) + ( n-1) \right) \\
	\end{aligned} \end{equation*}
	since the term 
		$$ C_1 ( T-t) r^{2 \lambda - 3} \left( \frac{1}{1 + Q_r^2} ( 2 \lambda - 1) ( 2 \lambda - 2) + (n-1) (2 \lambda - 1) + (n-1) \right) \ge 0$$
	is nonnegative for $n \ge 4$ and $k \ge 2$.
	Indeed, as also noted in appendix \ref{appendixConstants},
	\begin{gather*} 
		2 \lambda_k -1
		\ge 2 \lambda_2 -1	
		= 4 - 1 + \alpha -1	
		= 2 - | \alpha | 		
		\ge 0
	\end{gather*}
	and
	\begin{equation*} \begin{aligned}
	&\left( \frac{1}{1 + Q_r^2} ( 2 \lambda - 1) ( 2 \lambda - 2) + (n-1) (2 \lambda - 1) + (n-1) \right)	\\
	=& \frac{1}{1 + Q_r^2} ( 2 \lambda - 1)^2 + \left( n -1 - \frac{ 1}{1 + Q_r^2} \right) (2 \lambda - 1) + (n-1) \\
	\ge& ( n - 2)( 2 \lambda - 1) + (n-1) 		\\
	\ge& 0 
	\end{aligned} \end{equation*}
	By the definition of $C_1$, it follows that $v^+$ is a supersolution on the domain where $v^+ > 0$.
	
	Finally, we confirm that $v^+$ satisfies the claimed estimates when $\Gamma \gg 1$ is sufficiently large.
	Let $\ol{C} = \ol{C}_k$. At $r = \Gamma \sqrt{T - t}$,
	\begin{equation*} \begin{aligned}
		&v^+(r,t) - \ol{C} r^{2 \lambda + 1} \\
		=& ( C_0 - \ol{C} ) r^{2 \lambda + 1} - C_1 (T-t) r^{2 \lambda - 1} \\
		\ge& ( C_0 - \ol{C} ) r^{2 \lambda + 1} - C_1 ( T-t) ( \Gamma \sqrt{T-t})^{-2} r^{2 \lambda+1} \\
		=& \left( C_0 - \ol{C} - \frac{ C_1}{\Gamma^2} \right) r^{2 \lambda + 1}	\\
		\ge & 0 \\
	\end{aligned} \end{equation*}
	if $\Gamma \gg 1$ is sufficiently large, depending on $n,k, C_0$, such that
		$$C_0 - \ol{C} 
		\ge \frac{ C_1}{\Gamma^2} 
		= \frac{ \left[  ( 2 \lambda + 1) ( 2 \lambda) + (n-1)(2 \lambda + 1 ) + (n-1) \right] C_0 }{\Gamma^2}$$

	Additionally, since 
		$$C_0 - C_1 (T-t) r^{-2}$$
	is increasing in $r$ on the domain $r > 0$,
	the fact that $\ol{C} \ge 0$ and $v^+ \ge \ol{C} r^{2 \lambda + 1}$ at $r = \Gamma \sqrt{T-t}$ implies that
	$v^+(r,t) > 0$ for all $r \ge \Gamma \sqrt{T-t}$.
	In fact, $v^+(r,t) \ge \ol{C} r^{2 \lambda+1}$ for all $r \ge \Gamma \sqrt{T-t}$.
	Indeed,
	\begin{equation*} \begin{aligned}
		&v^+ - \ol{C} r^{2 \lambda + 1}	\\
		=& C_0 r^{2 \lambda + 1} - C_1(T-t) r^{2 \lambda - 1} - \ol{C} r^{2 \lambda +1}	\\
		=& r^{2 \lambda + 1} \left( C_0 - \ol{C} - C_1(T - t) r^{-2} \right)\\
		\ge& r^{2 \lambda + 1} \left( C_0 - \ol{C} - C_1  \Gamma^{-2} \right) \\
		\ge& 0
	\end{aligned} \end{equation*}
\end{proof}


\subsubsection{Interior Estimates} \label{intEstsOuterSubsubsec}

\begin{lem} \label{intEstsOuterLem}
	For any $C_0, \Gamma, M > 0$, there exists a constant $C_1 = C_1( n, k, \Gamma, M)$ such that
	if $v(r,t) = Q(r,t) -r$ solves (\ref{parabEqnPerturbation}) with
		$$0 \le v \le  C_0 r^{2 \lambda_k + 1}		\quad \text{ and } \quad 
		| 1 + v_r| = | Q_r|  \le  M						\qquad \text{for all }  r \ge \Gamma \sqrt{T- t}, \ t \in [0, T)$$
	then
		$$r | v_r | + r^2 | v_{rr} | \le  C_1 C_0 r^{2 \lambda_k+1}$$
	for all $\frac{15}{16} T < t  < T$ and  $2 \sqrt{2}  \Gamma \sqrt{ T - t} < r < \Gamma \sqrt{T}$.
\end{lem}
\begin{proof}
	The proof will employ a rescaling argument.
	Let $r_0, t_0$ be such that
		$$\frac{15}{16}T   < t_0 < T		\qquad \text{ and } \qquad 2 \sqrt{2} \Gamma \sqrt{ T - t_0} < r_0 < \Gamma \sqrt{T}$$
	Write $\lambda = \lambda_k$ and define
		$$W( \rho, \tau) \doteqdot r_0^{-2 \lambda - 1} v( r_0 \rho, t_0 + \tau r_0^2)$$
	Then $W$ solves
	\begin{equation} \label{parabEqnPerturbationRescaled}
		\partial_\tau W 
		= \frac{ W_{\rho \rho}}{1 + ( 1 + r_0^{2 \lambda} W_\rho)^2}
		+ \frac{ (n-1)}{\rho} W_\rho
		+ \frac{ (n-1) W}{\rho ( \rho + r_0^{2 \lambda} W) }
	\end{equation}
	
	If $\frac{1}{2} \le \rho \le \frac{3}{2}$ and $-\frac{1}{8 \Gamma^2}  \le \tau \le 0$
	then
		$$\sqrt{2} \Gamma \sqrt{ T - t_0} 
		\le r_0 \rho \le \frac{3}{2} \Gamma \sqrt{T}$$
	Moreover,
	\begin{gather*} 
		\frac{ \Gamma^2}{ \left( \frac{1}{4} - \Gamma^2 \frac{1}{8 \Gamma^2} \right) } ( T - t_0) 		
		= 8 \Gamma^2 ( T - t_0)
		\le r_0^2		\\
		\implies
		\Gamma^2 \left ( T - t_0 + \frac{1}{8 \Gamma^2} r_0^2 \right ) \le \frac{1}{4} r_0^2		\\
		\implies \Gamma \sqrt{ T - (t_0 + \tau r_0^2)} \le \Gamma \sqrt{ T - t_0 + \frac{1}{8 \Gamma^2} r_0^2} \le \frac{1}{2} r_0 \le r_0 \rho
	\end{gather*}
	and
	\begin{gather*}
		T 
		\ge t_0 + \tau r_0^2
		\ge \frac{15}{16} T - \frac{1}{8 \Gamma^2} \Gamma^2 T 
		> 0
	\end{gather*}
	Therefore,
	if $\frac{1}{2} \le \rho \le \frac{3}{2}$ and $-\frac{1}{8 \Gamma^2}\le \tau \le 0$,
	then $r = r_0 \rho$ and $t = t_0 + \tau r_0^2$ are in the domain where the assumed estimates on $v$ apply.
	In particular,
	\begin{gather*}
		0 \le W(\rho, \tau) \le r_0^{- 2\lambda - 1} C_0 ( r_0 \rho)^{2 \lambda + 1} \le C_0 \left( \frac{3}{2} \right)^{2 \lambda + 1} \\
		\text{ and } \qquad | 1 +  r_0^{2 \lambda}  W_\rho | =  | 1 + v_r | \le M
	\end{gather*}
	These estimates imply that the coefficients in equation (\ref{parabEqnPerturbationRescaled})
	are uniformly bounded for $ \frac{1}{2} \le \rho \le \frac{3}{2}$ and $- \frac{1}{8 \Gamma^2} \le \tau \le 0$
	with bounds depending only on $n,k$ and $M$.
	For example,
		$$0 \le \frac{2(n-1)}{\rho ( \rho + r_0^{2 \lambda} W)} \le \frac{ 2(n-1)}{\rho^2} \le 8(n-1)$$
	and
		$$\frac{1}{1+M^2} \le \frac{1}{ 1 + (1 + r_0^{2 \lambda} W_\rho)^2} \le 1$$
	Hence, interior estimates for parabolic equations (see e.g. ~\cite{LSU88}) imply that for some constant $C' = C'(n, k , \Gamma, M)$ depending only on $n, k, \Gamma, M$,
	\begin{equation*} \begin{aligned} 
		 r_0^{-2 \lambda +1} |v_{rr}(r_0, t_0) | + r_0^{- 2 \lambda} | v_r (r_0, t_0)|  
		&= | W_{\rho \rho} (0, 0) | + | W_{\rho} (0,0)| \\
		&\le C'   \sup_{(\rho, \tau) \in \left[ \frac{1}{2}, \frac{3}{2} \right] \times \left[ - \frac{1}{8 \Gamma^2} , 0 \right] } | W( \rho, \tau) | \\
		&= C' \sup_{(r, t) \in \left[ \frac{r_0}{2}, \frac{3 r_0}{2} \right] \times \left[ t_0 - \frac{1}{8 \Gamma^2} r_0^2, t_0 \right] } r_0^{-  2 \lambda -1} | v(r,t) | \\
		&\le C' \sup_{r \le \frac{3}{2} r_0 } r_0^{-  2 \lambda -1} C_0 r^{2 \lambda + 1} \\ 
		&= C' C_0 \left( \frac{3}{2} \right)^{2 \lambda + 1} 
	\end{aligned} \end{equation*}
	and so
		$$r_0 | v_r(r_0, t_0) | + r_0^2 | v_{rr}(r_0, t_0) |  \le  M' C_0 \left( \frac{3}{2} \right)^{2 \lambda + 1} r_0^{2 \lambda_k+1}$$
	Since $r_0, t_0$ were arbitrary, the statement of the lemma follows with $C_1  = C' \left ( \frac{3}{2} \right)^{2 \lambda + 1}$.
\end{proof}

The estimates of lemma \ref{intEstsOuterLem} yield a bound on the mean curvature.
\begin{cor} \label{outerParabOverlapEst}
	Let $n \ge 4$ and $k \ge 2$.
	For any $C_0, \Gamma, M > 0$, 
	there exists a constant $C_1 = C_1( n, k, T, \Gamma, M, C_0) < \infty$ such that
	if $v(r,t) = Q(r,t) -r$ solves (\ref{parabEqnPerturbation}) with
		$$0 \le v \le  C_0 r^{2 \lambda_k + 1}		\quad \text{ and } \quad 
		| 1 + v_r| = | Q_r|  \le  M						\qquad \text{for all }  r \ge \Gamma \sqrt{T- t}, \ t \in [0, T)$$
	then
		$$|H_{\Sigma(t)}| \le C_1$$
	for all
		$\frac{15}{16} T < t < T$  and $2 \sqrt{2} \Gamma \sqrt{ T - t} < | \mathbf x |  < \Gamma \sqrt{T}$.
\end{cor}
\begin{proof}
	Recall
	\begin{equation*} \begin{aligned}
		| H_{\Sigma(t)} |	
		&= \frac{1}{ \sqrt{1 + Q_r^2} } \left| \frac{Q_{rr}}{1 + Q_r^2} + \frac{n-1}{r} Q_r - \frac{n-1}{Q} \right| \\
		&= \frac{1}{ \sqrt{1 + Q_r^2} } \left| \frac{v_{rr}}{1 + Q_r^2} + \frac{ (n-1)}{r} v_r - \frac{n-1}{r} \left( \frac{1}{ 1 +\frac{v}{r} } - 1 \right) \right| \\
		& \le  | v_{rr} | + (n-1) \frac{ | v_r |}{r} + \frac{n-1}{r} \left| \frac{1}{ 1 +\frac{v}{r} } - 1 \right| \\
	\end{aligned} \end{equation*} 
	For $\frac{15}{16} T < t < T$ and $2 \sqrt{2} \Gamma \sqrt{T - t} < r < \Gamma \sqrt{T}$,
	lemma \ref{intEstsOuterLem} implies
		$$| v_{rr} | \lesssim_{n,k,\Gamma,M, C_0} r^{2 \lambda_k -1} 	\qquad 
		| v_r| \lesssim_{n,k,\Gamma,M, C_0} r^{2 \lambda_k} 	\qquad 
		\left| \frac{v}{r} \right| \le C_0 r^{2 \lambda_k} \le C_0 \left( \Gamma \sqrt{T} \right)^{2 \lambda_k}$$
	Because $x \mapsto \frac{1}{1 +x}$ is locally Lipschitz,
		$$\left| \frac{1}{ 1 +\frac{v}{r} } - 1 \right| \lesssim_{k, T, \Gamma, C_0} \left| \frac{v}{r} \right| \le C_0 r^{2 \lambda_k}$$
	Noting that $2 \lambda_k -1 \ge 0$, it now follows that
		$$| H_{\Sigma(t)} | \lesssim_{n,k, T, \Gamma, M, C_0} r^{2 \lambda_k - 1} \le \left( \Gamma \sqrt{T} \right)^{2 \lambda_k - 1} < \infty$$
	for all $\frac{15}{16} T < t < T$ and $2 \sqrt{2} \Gamma \sqrt{T - t} < | \mathbf x| < \Gamma \sqrt{T}$.
\end{proof}

We may now present the proof of the main theorem.
\begin{proof}
	(of theorem \ref{mainThmFull})
	Let $n \ge 4$ and let $k > 2$ be an even integer.
	Note that it suffices to bound the mean curvature for times $t$ in a neighborhood of the singular time $T$.
	In the notation of lemma \ref{supersolnLem}, let $C_0 = \ol{C}_k+1$ 
	and take $\Gamma_0 \gg 1$ sufficiently large such that 
	$$v^+(r,t) > \ol{C}_k r^{2 \lambda_k +1} 		\qquad \text{for all } r = \Gamma_0 \sqrt{ T - t}, \ t \in [0, T)$$
	By remark \ref{adjustInitData}, we may assume without loss of generality that the initial profile function $Q(r,0)$ is taken such that
	\begin{equation*} \begin{aligned}
		&r \le Q(r, 0) \le r + v^+(r,0)	\qquad & \text{for all } r \ge \Gamma_0 \sqrt{T}, \\
		& 1 \le Q_r(r,0) \le M_0 < \infty			&	\text{for all } r \ge \Gamma_0 \sqrt{T},  \\
		\text{and }& \lim_{r \to \infty} Q(r,0) - r = \infty		&	\\
	\end{aligned} \end{equation*}
	for some constant $M_0$.
	By lemma \ref{outerRegEst},
	there exists $\Upsilon_0 \ge \frac{12}{5} \Gamma_0$ sufficiently large depending on $M_0$ and $n$ 
	such that the mean curvature is uniformly bounded for $r \ge \Upsilon_0 \sqrt{T}$ and $t \in \left[ \frac{T}{2} , T \right)$.
	
	The coarse barrier arguments in subsection \ref{coarseBarriers} now imply
	\begin{equation*} \begin{aligned}
		&r \le Q(r, t) 	 & \text{for all } r \ge \Gamma_0 \sqrt{T-t} \text{ and } t \in [0, T), \\
		\text{and } &  |Q_r(r,t)| \le M_1 < \infty	\qquad		&	\text{for all } r \ge \Gamma_0 \sqrt{T-t} \text{ and } t \in [0, T).  \\
	\end{aligned} \end{equation*}
	for some constant $M_1 \ge M_0$.
	Lemma \ref{supersolnLem} additionally implies that
		$$Q(r,t) \le r + v^+(r,t)		\qquad \text{for all } r \ge \Gamma_0 \sqrt{T-t}, \ t \in [0, T).$$
	In particular, since $\Upsilon_0 \ge \frac{12}{5} \Gamma_0 > \Gamma_0$, these estimates all hold for $r \ge \Upsilon_0 \sqrt{ T - t}$.
	Corollary \ref{outerParabOverlapEst} with $\Gamma = \Upsilon_0$ now applies to give uniform mean curvature bounds
	on the domain $r \ge 2 \sqrt{2} \Upsilon_0 \sqrt{T - t}$ and $t \in \left[ \frac{15}{16} T, T \right)$.
	
	By remark \ref{paramsThatWork}, there exists $a \in \left( | \alpha|, | \alpha| + 1 \right)$ such that
		$$\lambda_k \left( 1 - \frac{a}{ 1 + | \alpha| } \right) - \frac{1}{2} \ge 0$$
	Finally, corollary \ref{bddHInner+Parab2} with $\Gamma = 2 \sqrt{2} \Upsilon_0$ uniformly bounds the mean curvature on the remainder of the evolving hypersurface.
\end{proof}

\subsection{Compactifying the Vel{\'a}zquez Mean Curvature Flow Solutions} \label{compactifying}

To conclude, we sketch how one might construct \textit{closed} mean curvature flow solutions that exhibit the same dynamics as the Vel{\'a}zquez mean curvature flow solutions.
Vel{\'a}zquez's existence result follows from a topological degree argument applied to a suitably defined disks' worth of initial hypersurfaces (denoted $\Gamma_\eta(\alpha)$ in section 3 of ~\cite{V94}).
For $0 < R_1 < R_2$ sufficiently large and $C>0$ sufficiently large, adjust the profile functions 
 $Q(r,0)$ of these initial hypersurfaces in the region where $r \ge R_1$ 
 so that the initial hypersurfaces remain smooth but
	$$Q(r,0) = \sqrt{2(2n-1)C -r^2}		\qquad \text{for } R_2 \le r \le \sqrt{2(2n-1)C}$$
Such a choice is motivated by the profile function for the shrinking spheres as given in subsubsection \ref{shrinkingSphere}.

For suitably chosen $R_1, R_2,$ and $C$, 
one may again argue as in lemma \ref{outerRegEst} and use the interior estimates of ~\cite{EH91}
to control the solution in the outer region up to say time $t=1$.
One must then check that such estimates in the outer region suffice to reproduce the results in section 4 of ~\cite{V94}.
Indeed, the exponential weight $e^{-r^2/4}$ that appears in the integral estimates within section 4 of ~\cite{V94} 
suggests that the results in section 4 of ~\cite{V94} will continue to hold with this new assignment of initial data.
For Ricci flow, ~\cite{Stolarski19} proves similar estimates to obtain closed Ricci flow solutions analogous to Vel{\'a}zquez's mean curvature flow solutions. 
The remaining arguments in ~\cite{V94} and this paper then follow almost immediately
to give the existence of closed mean curvature flow solutions in $\mathbb{R}^N$ ($N \ge 8$)
that become singular at some time $T < 1$ and have uniformly bounded mean curvature.

\appendix

\section{Computations for $O(n) \times O(n)$-Invariant Hypersurfacees} \label{appendixComputations}

Consider parametrizing an $O(n) \times O(n)$-invariant hypersurface $\Sigma^{2n-1}$ in $\mathbb{R}^{n } \times \mathbb{R}^n$ by
	$$F : \mathbb{R}^n \times \mathbb{S}^{n-1} \to \mathbb{R}^n \times \mathbb{R}^n$$
	$$F( \mathbf x, \mathbf \theta) = ( \mathbf x, Q( | \mathbf x|) \mathbf \theta) $$
We will occasionally write
	$$\mathbf y = Q( | \mathbf x|) \theta		\qquad | \mathbf y | = Q( | \mathbf x |)$$
Recall too that $r = | \mathbf x|$ and $' = \frac{\partial}{\partial r}$ denotes derivatives with respect to $r$.
\begin{prop} \label{tangentVecs}
	The tangent vectors to the hypersurface are
		$$\frac{ \partial F}{\partial \theta^i} 
		= \left( \mathbf 0, Q(| \mathbf x|) \frac{ \partial }{\partial \theta^i} \right) 
		= \mathbf 0 \oplus Q( |\mathbf x |) \frac{ \partial }{\partial \theta^i}$$
		$$\frac{ \partial F}{\partial x^i} 
		= \frac{ \partial}{\partial x^i} \oplus Q'( |\mathbf x|) \frac{ x^i }{ | \mathbf x |} \theta
		= \frac{ \partial}{\partial x^i} \oplus \frac{Q'( |\mathbf x|)}{ Q( | \mathbf x | )} \frac{ x^i }{ | \mathbf x |} \mathbf y$$
\end{prop}

\begin{prop} \label{unitNormal}
	The unit normal $\nu_\Sigma$ to $\Sigma$ at the point $( \mathbf x , \mathbf y ) = ( \mathbf x, Q( | \mathbf x |) \theta) = F( \mathbf x, \theta)$ is
		$$ \nu_\Sigma ( \mathbf x, \mathbf y ) 
		= \frac{ \left( - Q'( | \mathbf x |) \frac{ \mathbf x}{ |\mathbf x |}, \frac{ \mathbf y}{ | \mathbf y |} \right)}{ \sqrt{ 1 + Q'( | \mathbf x |)^2} }
		= \frac{ \left( - Q'( | \mathbf x |) \frac{ \mathbf x}{ |\mathbf x |}, \theta \right)}{ \sqrt{ 1 + Q'( | \mathbf x |)^2} }$$
		
	In particular,
		$$ \nu_\Sigma( \mathbf x, \mathbf y ) \cdot ( \mathbf x, \mathbf y) 
		= \frac{ - Q'( | \mathbf x |) | \mathbf x | + Q( | \mathbf x |) }{\sqrt{ 1 + Q' ( | \mathbf x |)^2 }}$$
\end{prop}
\begin{proof}
	Clearly, $\nu_\Sigma$ has norm one everywhere.
	It thus suffices to check that \\
	$\left( - Q'( | \mathbf x |) \frac{ \mathbf x}{ |\mathbf x |}, \frac{ \mathbf y}{ | \mathbf y |} \right)$ is orthogonal to the tangent vectors from proposition \ref{tangentVecs}.
	\begin{gather*}
		\left( - Q'( | \mathbf x |) \frac{ \mathbf x}{ |\mathbf x |}, \frac{ \mathbf y}{ | \mathbf y |} \right) 
		\cdot \frac{ \partial F}{ \partial \theta^i} 
		= 0 +  Q( | \mathbf x |) \left( \theta \cdot  \frac{ \partial }{ \partial \theta^i} \right)
		= 0
	\end{gather*}
	since $\theta \in \mathbb{S}^{n-1}$ is orthogonal to any tangent vector $\frac{ \partial}{ \partial \theta^i} \in T_\theta \mathbb{S}^{n-1}$.
	
	\begin{gather*}
		\left( - Q'( | \mathbf x |) \frac{ \mathbf x}{ |\mathbf x |}, \frac{ \mathbf y}{ | \mathbf y |} \right) 
		\cdot \frac{ \partial F}{ \partial x^i} 
		= - Q' \frac{ x^i}{| \mathbf x |} + \frac{ Q'}{| \mathbf y |} \frac{ x^i | \mathbf y |^2}{ | \mathbf x | | \mathbf y |}
		= 0
	\end{gather*}

	For the second statement, observe
	\begin{equation*} \begin{aligned}
		 \nu_\Sigma \cdot ( \mathbf x, \mathbf y) 
		= \frac{1}{ \sqrt{ 1 + Q'^2} } \left ( - Q' \frac{ \mathbf x }{ | \mathbf x |},  \frac{ \mathbf y }{ | \mathbf y |} \right) 
		\cdot ( \mathbf x, \mathbf y) 
		= \frac{1}{ \sqrt{ 1 + Q'^2} } \left[ - | \mathbf x | Q'   + Q  \right]
	\end{aligned} \end{equation*}
	using that $| \mathbf y | = Q( | \mathbf x |)$.
\end{proof}

\begin{cor}
	$$\mathbf F \cdot  \nu_\Sigma = \frac{ Q - | \mathbf x | Q' }{ \sqrt{ 1 + Q'^2} }$$
\end{cor}
\begin{proof}
	The proof is a straightforward computation using the previous proposition
	\begin{equation*} \begin{aligned}
		\mathbf F \cdot  \nu_\Sigma
		&= ( \mathbf x , Q( | \mathbf x |) \mathbf \theta ) \cdot \frac{  ( - Q' \frac{ \mathbf x }{ | \mathbf x |} , \mathbf \theta ) }{ \sqrt{ 1 + Q'^2} } \\	
		&= \frac{ - Q' \mathbf x \cdot \frac{ \mathbf x }{| \mathbf x |} + Q \theta \cdot \theta}{ \sqrt{ 1 + Q'^2} }	\\
		&= \frac{ Q - | \mathbf x | Q' }{ \sqrt{ 1 + Q'^2} }  	
	\end{aligned} \end{equation*}

\end{proof}

\begin{prop} \label{RiemMetric}
	The induced metric $g_{\Sigma}$ on the hypersurface $\Sigma$ has components
	\begin{equation*} \begin{aligned} 
		\frac{ \partial F}{\partial x^i } \cdot \frac{ \partial F}{\partial x^j } 
		&=  \delta_{ij} + Q' (| \mathbf x |)^2 \frac{ x^i x^j}{ | \mathbf x |^2} 		\\
		\frac{ \partial F}{ \partial \theta^i } \cdot \frac{ \partial F}{ \partial \theta^j} 
		&= Q( | \mathbf x |)^2 \frac{ \partial }{ \partial \theta^i} \cdot \frac{ \partial }{ \partial \theta^j}		\\
		 \frac{ \partial F}{ \partial x^i } \cdot \frac{ \partial F}{ \partial \theta^j} &= 0 
	\end{aligned} \end{equation*}
	\begin{equation*}
		g_{\Sigma}
		= \left( \begin{array}{cc}
			 \delta_{ij} + Q'( |\mathbf x|)^2 \frac{ x^i x^j}{ | \mathbf x |^2}  & 0 \\
			0 & Q( | \mathbf x |)^2 g_{\mathbb{S}} \\
		\end{array} \right)
	\end{equation*}
	In particular, at a point $\mathbf x = ( x^1, 0, ... , 0)$, $g_\Sigma$ has the block decomposition
	\begin{equation*}
	g_{\Sigma}
		= \left( \begin{array}{ccc}
			 1 + Q'(| \mathbf x |)^2 & 0 & 0 \\
			0 & Id_{n-1} & 0 \\
			0 & 0 & Q( | \mathbf x |)^2 g_{\mathbb{S}} \\
		\end{array} \right)
	\end{equation*}
	In polar coordinates for $\mathbf x = (r, \omega)$,
	\begin{equation*}
	g_{\Sigma}
		= \left( \begin{array}{ccc}
			 1 + Q'(r)^2 & 0 & 0 \\
			0 & r^2 g_{\mathbb{S}} & 0 \\
			0 & 0 & Q( r )^2 g_{\mathbb{S}} \\
		\end{array} \right)
	\end{equation*}
\end{prop}
\begin{proof}
	We use proposition \ref{tangentVecs} throughout.
	\begin{gather*}
		\frac{ \partial F}{ \partial \theta^i} \cdot \frac{ \partial F}{ \partial \theta^j} 
		= Q^2 \frac{ \partial }{ \partial \theta^i } \cdot \frac{ \partial }{ \partial \theta^j }
	\end{gather*}
	
	\begin{equation*} \begin{aligned}
		\frac{ \partial F}{ \partial x^i} \cdot \frac{ \partial F}{ \partial x^j} 
		&= \left( \frac{ \partial}{ \partial x^i } \oplus Q' \frac{ x^i }{ | \mathbf x |} \theta \right) 
		\cdot \left( \frac{ \partial}{ \partial x^j } \oplus Q' \frac{ x^j }{ | \mathbf x |} \theta \right) \\
		&=  \frac{ \partial }{ \partial x^i} \cdot \frac{ \partial }{ \partial x^j} + Q'^2 \frac{ x^i x^j}{ | \mathbf x |^2} | \theta|^2 \\
		&=  \delta_{ij} + Q'^2 \frac{ x^i x^j}{ | \mathbf x |^2} 
	\end{aligned} \end{equation*}
	
	Finally,
	\begin{gather*}
		\frac{ \partial F}{ \partial x^i } \cdot \frac{ \partial F}{ \partial \theta^j}
		= \left( \frac{ \partial}{ \partial x^i } \oplus Q' \frac{ x^i }{ | \mathbf x |} \theta \right) 
		\cdot \left( \mathbf 0 \oplus Q \frac{ \partial}{ \partial \theta^i} \right)
		= Q' Q \frac{ x^i}{ | \mathbf x | } \theta \cdot \frac{ \partial }{ \partial \theta^i} 
		= 0 		
	\end{gather*}
	since tangent vectors $\frac{ \partial }{ \partial \theta^i}$ to the sphere $\mathbb{S}^{n-1}$ are orthogonal to the position $\theta$ on the sphere. 
	
	The first block decomposition of $g_{\Sigma}$ is immediate.
	For the second, note that at $\mathbf x = ( x^1, 0, ... 0)$, $| \mathbf x |^2 = ( x^1)^2$.
	Hence, 
		$$ \delta_{ij} + Q'^2 \frac{ x^i x^j}{| \mathbf x |^2}  = 
		\left\{ \begin{array}{cc}
			 1 + Q'^2, & i = j = 1 \\
			\delta_{ij}, & \text{ else} \\
		\end{array} \right.$$
	The resulting block decomposition follows.
	
	The final block decomposition for $g_{\Sigma}$ follows from the prior and $O(n) \times O(n)$-invariance.
\end{proof}

\begin{cor}
	At a point $\mathbf x = ( x^1, 0, ... 0)$, the inverse matrix of $g_\Sigma$ is
		$$ g_\Sigma^{-1}
		= \left( \begin{array}{ccc} 
			\frac{1}{  1 + Q'(| \mathbf x |)^2 } & 0 & 0 \\
			0 & Id_{n-1} & 0 \\
			0 & 0 & \frac{1}{Q( | \mathbf x |)^2 } g_{\mathbb{S}}^{-1} \\
		\end{array} \right)$$
\end{cor}

\begin{prop} \label{2ndFundlForm}
	The second fundamental form is
	\begin{equation*}	\begin{aligned}
		A_{ij} &= - \sigma_j \cdot \nabla_{\sigma_i} \nu_\Sigma = \nabla_{\sigma_i} \sigma_j \cdot \nu_{\Sigma}	\\
		&= \frac{1}{ \sqrt{ 1 + Q'( | \mathbf x | )^2} }   
		\left( \begin{array}{cc}
			Q''( |\mathbf x |) \frac{ x^i  x^j}{| \mathbf x |^2} + Q'( |\mathbf x |) \frac{ \delta_{ij} | \mathbf x |^2 - x^i x^j }{ | \mathbf x |^3} & 0 \\
			0 & - Q( | \mathbf x | ) \,  (g_{\mathbb {S}^{n-1} })_{ij}	\\
		\end{array} \right)
	\end{aligned} \end{equation*}
	
	In particular, at a point where $\mathbf x = ( x^1, 0, ... , 0)$, the second fundamental form becomes
		$$A_{ij} 
		= \frac{1}{ \sqrt{ 1 + Q'( | \mathbf x | )^2} }   
		\left( \begin{array}{ccc}
			Q''( |\mathbf x |) & 0 & 0 \\
			0 & \frac{ Q'( | \mathbf x |) }{ | \mathbf x |} \delta_{ij} & 0 \\
			0 & 0 & - Q( | \mathbf x | ) \,  (g_{\mathbb {S}^{n-1} })_{ij}	\\
		\end{array} \right)$$
\end{prop}
\begin{proof}
	For $\sigma_i = \frac{ \partial F}{ \partial x^i}$, it follows that
	\begin{equation*} \begin{aligned}
		\nabla_{\sigma_i} \nu_\Sigma
		&= \frac{ \partial }{ \partial x^i} \nu_\Sigma ( F( \mathbf x, \theta) ) \\
		&= \frac{ \partial }{ \partial x^i} \left( \frac{ 1}{ \sqrt{ 1 + Q'^2} } \right) 
		\left( - Q'( | \mathbf x |) \frac{ \mathbf x}{ |\mathbf x |}, \theta  \right) 
		+ \frac{ 1}{ \sqrt{ 1 + Q'^2} } 
		\left( - \frac{ \partial }{ \partial x^i} \left( Q' \frac{ \mathbf x }{ | \mathbf x |} \right), \frac{ \partial }{ \partial x^i} \theta \right) \\
		&= \frac{ \partial }{ \partial x^i} \left( \frac{ 1}{ \sqrt{ 1 + Q'^2} } \right) 
		\left( - Q'( | \mathbf x |) \frac{ \mathbf x}{ |\mathbf x |}, \theta  \right) 
		+ \frac{ 1}{ \sqrt{ 1 + Q'^2} } 
		\left( - \frac{ \partial }{ \partial x^i} \left( Q' \frac{ \mathbf x }{ | \mathbf x |} \right), \mathbf 0 \right)
	\end{aligned} \end{equation*}
	The first term is a multiple of $\nu_\Sigma$ so it suffices to compute the second term.
	The second term equals
	\begin{equation*} 
		\frac{ 1}{ \sqrt{ 1 + Q'^2} } 
		\left( - \frac{ \partial }{ \partial x^i} \left( Q' \frac{ \mathbf x }{ | \mathbf x |} \right), \mathbf 0 \right)	\\
		= \frac{ -1}{ \sqrt{ 1+ Q'^2} }
		\left( Q'' \frac{ x^i}{ | \mathbf x |} \frac{ \mathbf x }{ | \mathbf x |} 
		+ \frac{ Q'}{ | \mathbf x |} \frac{ \partial }{ \partial x^i} 
		- Q' \frac{ x^i}{ | \mathbf x |^3} \mathbf x , 
		\mathbf 0 \right)
	\end{equation*}

	It now follows that
	$$\sigma_j \cdot \nabla_{ \sigma_i} \nu_\Sigma 
	= \frac{ - 1}{ \sqrt{ 1 + Q'^2} }  \, \sigma_j \cdot  \left( Q'' \frac{ x^i}{ | \mathbf x |} \frac{ \mathbf x }{ | \mathbf x |} 
		+ \frac{ Q'}{ | \mathbf x |} \frac{ \partial }{ \partial x^i} 
		- Q' \frac{ x^i}{ | \mathbf x |^3} \mathbf x , 
		\mathbf 0 \right)$$
		
	If $\sigma_j = \frac{ \partial F}{ \partial \theta^j}$ then
	\begin{equation*} \begin{aligned}
		\frac{ \partial F}{ \partial \theta^j} \cdot \nabla_{\sigma_i} \nu_\Sigma 		
		&= \frac{ - 1}{ \sqrt{ 1 + Q'^2} } \left( \mathbf 0, Q \frac{ \partial}{ \partial \theta^j}  \right) \cdot \left( Q'' \frac{ x^i}{ | \mathbf x |} \frac{ \mathbf x }{ | \mathbf x |} 
		+ \frac{ Q'}{ | \mathbf x |} \frac{ \partial }{ \partial x^i} 
		- Q' \frac{ x^i}{ | \mathbf x |^3} \mathbf x , \mathbf 0 \right) 	\\
		&= 0
	\end{aligned} \end{equation*}
		
	If $\sigma_j = \frac{ \partial F}{ \partial x^j}$ then
	\begin{equation*} \begin{aligned}	
		\frac{ \partial F}{ \partial x^j} \cdot \nabla_{\sigma_i} \nu_\Sigma
		&= \frac{ -1}{ \sqrt{ 1 + Q'^2} } \frac{ \partial }{ \partial x^j } \cdot 
		\left( Q'' \frac{ x^i}{ | \mathbf x |} \frac{ \mathbf x }{ | \mathbf x |} 
		+ \frac{ Q'}{ | \mathbf x |} \frac{ \partial }{ \partial x^i} 
		- Q' \frac{ x^i}{ | \mathbf x |^3} \mathbf x \right)			\\
		&= \frac{ -1}{ \sqrt{ 1 + Q'^2} } 
		\left[ Q'' \frac{ x^i x^j }{ | \mathbf x |^2} 
		+ \frac{ Q'}{ | \mathbf x |} \delta_{ij}
		- Q' \frac{ x^i x^j}{ | \mathbf x |^3} \right]
	\end{aligned} \end{equation*}
		
	Now let $\sigma_i = \frac{ \partial F}{ \partial \theta^i}$.
	It follows that
	\begin{equation*} \begin{aligned}
		\nabla_{\sigma_i} \nu_\Sigma 
		&= \frac{ \partial }{ \partial \theta^i} \nu_\Sigma ( F( \mathbf x, \theta) ) \\
		&= \frac{ 1}{ \sqrt{ 1 + Q'^2} } \left( \frac{ \partial }{ \partial \theta^i} \left( - Q' \frac{ \mathbf x }{ | \mathbf x |} \right) , 
		\frac{ \partial \theta }{ \partial \theta^i} \right) \\
		&= \frac{ 1}{ \sqrt{ 1 + Q'^2} } \left( \mathbf 0, \frac{ \partial \theta }{ \partial \theta^i } \right)
	\end{aligned} \end{equation*}
	
	If $\sigma_j = \frac{ \partial F}{ \partial x^j}$ then
	\begin{gather*}
		\sigma_j \cdot \nabla_{\sigma_i} \nu_\Sigma
		= \left( \frac{ \partial }{ \partial x^j} , Q' \frac{ x^j}{ | \mathbf x |} \theta \right)
		\cdot \frac{ 1}{ \sqrt{ 1 + Q'^2} } \left( \mathbf 0, \frac{ \partial \theta }{ \partial \theta^i } \right)
		= 0
	\end{gather*}
	as expected from symmetry of the second fundamental form.
	
	If $\sigma_j = \frac{ \partial F}{ \partial \theta^j}$ then
	\begin{equation*} \begin{aligned}
		\sigma_j \cdot \nabla_{\sigma_i} \nu_\Sigma
		&= \left( \mathbf 0, Q \frac{ \partial \theta}{ \partial \theta^j} \right) 
		\cdot  \frac{ 1}{ \sqrt{ 1 + Q'^2} } \left( \mathbf 0, \frac{ \partial \theta }{ \partial \theta^i } \right) \\
		&= \frac{ Q}{ \sqrt{ 1 + Q'^2} } \frac{ \partial \theta}{ \partial \theta^j} \cdot  \frac{ \partial \theta}{ \partial \theta^i}\\
		&= \frac{ Q}{ \sqrt{1 + Q'^2} } ( g_{\mathbb{S}} )_{ij} 
	\end{aligned} \end{equation*}
\end{proof}

\begin{prop} \label{meanCurv}
	The (scalar) mean curvature $H$ is
		$$H = \frac{1}{ \sqrt{1 + Q'( |\mathbf x |)^2} } \left( \frac{ Q''( |\mathbf x |)}{ 1 + Q'( |\mathbf x |)^2} + \frac{ (n-1) Q'( |\mathbf x |)}{| \mathbf x |} - \frac{ n-1}{ Q( |\mathbf x |)} \right)$$
\end{prop}
\begin{proof}
	We compute the mean curvature at a point of the form $ \mathbf x = ( x^1, 0, ... , 0)$ and then use the $O(n) \times O(n)$ symmetry of $\Sigma$.
	\begin{equation*} \begin{aligned}
		&g_{\Sigma}^{-1} A \\
		=& \frac{ 1}{ \sqrt{ 1 + Q'^2} }
		\left( \begin{array}{ccc} 
			\frac{1}{  1 + Q'(| \mathbf x |)^2 } & 0 & 0 \\
			0 & Id_{n-1} & 0 \\
			0 & 0 & \frac{1}{Q( | \mathbf x |)^2 } g_{\mathbb{S}}^{-1} \\
		\end{array} \right)	\\
		& \cdot \left( \begin{array}{ccc}
			Q''( |\mathbf x |) & 0 & 0 \\
			0 & \frac{ Q'( | \mathbf x |) }{ | \mathbf x |} Id_{n-1} & 0 \\
			0 & 0 & - Q( | \mathbf x | ) \,  g_{\mathbb {S}^{n-1} }	\\
		\end{array} \right) \\
		=& \frac{ 1 }{ \sqrt{ 1 + Q'^2} }
		\left( \begin{array}{ccc}
			\frac{ Q''}{ 1 + Q'^2} & 0 & 0 \\
			0 & \frac{ Q'}{ | \mathbf x | } Id_{n-1} & 0 \\
			0 & 0 & \frac{ -1}{ Q} Id_{n-1} \\
		\end{array} \right) 
		\\
	\end{aligned} \end{equation*}	
	\begin{equation*}
		\implies
		H = tr( g_{\Sigma}^{-1} A ) 
		= \frac{ 1}{ \sqrt{ 1 + Q'^2} } \left( \frac{ Q''}{ 1 + Q'^2} + \frac{ (n-1) Q'}{| \mathbf x |} - \frac{ n-1}{ Q} \right) 
	\end{equation*}
	Because $\Sigma$ is $O(n) \times O(n)$ invariant, this formula holds throughout $\Sigma$.
\end{proof}

\begin{prop} \label{MCFforQ}
	Mean curvature flow of $\Sigma$ corresponds to the following partial differential equation for $Q$
		$$\partial_t Q = \frac{ Q''}{ 1 + Q'^2} + \frac{ (n-1)}{| \mathbf x |} Q' - \frac{ n-1}{Q}$$
\end{prop}
\begin{proof}
	The mean curvature flow equation can be written as
		$$\nu_{\Sigma} \cdot \partial_t F ( \mathbf x, \theta, t)  = H$$
	For
		$$F( \mathbf x, \theta, t) = ( \mathbf x , Q( | \mathbf x |, t) \theta )$$
	the left-hand side becomes
	\begin{gather*}
		\nu_{\Sigma} \cdot \partial_t F ( \mathbf x, \theta, t) 
		= \frac{ 1}{ \sqrt{ 1 + Q'^2} } \left( - Q' \frac{ \mathbf x }{ | \mathbf x | } , \theta \right) 
		\cdot 
		( \mathbf 0 , ( \partial_t Q) \theta) 
		= \frac{ ( \partial_t Q) \theta \cdot \theta} { \sqrt{ 1+ Q'^2} }
		= \frac{ \partial_t Q}{ \sqrt {1 + Q'^2}}
	\end{gather*}
	Cancelling the $(  1 + Q'^2)^{-1/2} $ from both sides of the partial differential equation yields the desired equation.
\end{proof}

\begin{remark}
	The above partial differential equation differs from that in \cite{V94} by a factor of $2$ on the right-hand side.
	This factor is due to the fact that \cite{V94} parametrizes the hypersurfaces in terms of the sphere of radius $\frac{1}{\sqrt{2}}$ instead of the unit sphere $\mathbb{S}^{n-1}$,
	and so the profile functions here correspond to
		$$\frac{ 1}{ \sqrt{2} } Q \left( \sqrt{ 2} | \mathbf x | , t \right)$$
	in the notation of \cite{V94} 
\end{remark}

\begin{prop} \label{2ndFundlFormNorm}
	$$| A|^2 = A_j^i A_i^j 
	= \frac{1}{  1 + Q'^2} \left[ \left( \frac{ Q''}{1 + Q'^2} \right)^2 + (n-1) \frac{ Q'^2}{| \mathbf x |^2} + \frac{ n-1}{Q^2} \right]$$
\end{prop}
\begin{proof}
	In proposition \ref{meanCurv}
	we computed that 
	\begin{equation*}
		g_\Sigma^{-1} A =
		g^{ik}A_{kj}
		= A^i_j 
		= 
		\frac{ 1 }{ \sqrt{ 1 + Q'^2} }
		\left( \begin{array}{ccc}
			\frac{ Q''}{ 1 + Q'^2} & 0 & 0 \\
			0 & \frac{ Q'}{ | \mathbf x | } Id_{n-1} & 0 \\
			0 & 0 & \frac{ -1}{ Q} Id_{n-1} \\
		\end{array} \right) 
	\end{equation*}
	at points $\mathbf x = ( x^1, 0, ... , 0)$.
	Hence,
	\begin{equation*} \begin{aligned}
		| A |^2
		&= tr ( A^i_k A^k_j ) \\
		&= tr
		\left( 		\frac{ 1 }{  1 + Q'^2 }
		\left( \begin{array}{ccc}
			\left( \frac{ Q''}{ 1 + Q'^2} \right)^2 & 0 & 0 \\
			0 & \frac{ Q'^2}{ | \mathbf x |^2 } Id_{n-1} & 0 \\
			0 & 0 & \frac{ 1}{ Q^2} Id_{n-1} \\
		\end{array} \right) \right) \\
		&= \frac{ 1}{ 1 + Q'^2} 
		\left[ \left(\frac{ Q''}{ 1 + Q'^2} \right)^2 + (n-1) \frac{ Q'^2}{ | \mathbf x |^2} + \frac{ n-1}{ Q^2} \right]
	\end{aligned} \end{equation*}
	By $O(n) \times O(n)$-invariance of $\Sigma$, this formula holds throughout $\Sigma$.
\end{proof}

\begin{remark} \label{A^2Cone}
	On the Simons cone $Q = | \mathbf x |$, proposition \ref{2ndFundlFormNorm} implies
		$$| A|^2 = \frac{ n-1}{ | \mathbf x |^2} $$
\end{remark}

\begin{prop}
	The Laplace-Beltrami operator $\Delta_\Sigma$ for an $O(n) \times O(n)$-invariant hypersurface
	$( \Sigma, g_\Sigma)$
	acting on $u = u( | \mathbf x|)$ is given by
		$$\Delta_\Sigma u = \frac{ 1}{1 + Q'^2} \left[ u'' + \frac{ n-1}{ | \mathbf x |} u' - \frac{  Q' Q''}{ 1 + Q'^2} u' +  ( n-1) \frac{ Q'}{ Q} u'\right]$$
\end{prop}
\begin{proof}
	Recall that
		$$\Delta_\Sigma u = \frac{1}{ \sqrt{ det \, g}} \partial_i \left( \sqrt{det \, g} \, g^{ij} \partial_j u \right)$$
	Note that at $\mathbf x = (x^1, 0, ... , 0)$
	\begin{equation*} \begin{aligned}
		g^{-1} \partial u
		= 
		\left( \begin{array}{ccc} 
			\frac{1}{  1 + Q'(| \mathbf x |)^2 } & 0 & 0 \\
			0 & Id_{n-1} & 0 \\
			0 & 0 & \frac{1}{Q( | \mathbf x |)^2 } g_{\mathbb{S}}^{-1} \\
		\end{array} \right)
		\left( \begin{array}{c}
			u'( |\mathbf x | ) \frac{ x^1}{ | \mathbf x |} \\
			0 \\
			0 \\
		\end{array} \right) 
		= 
		\frac{ u'(| \mathbf x |) }{1 + Q'^2} \frac{ x^1}{ | \mathbf x |}
	\end{aligned} \end{equation*}
	By symmetry it follows that
	\begin{equation*} \begin{aligned}
		\Delta_\Sigma u 
		=& \frac{ 1}{ \sqrt{ det g} } \sum_{i = 1}^{n} \frac{ \partial }{ \partial x^i} 
		\left( \sqrt{ det g } \frac{ u'(| \mathbf x |) }{1 + Q'^2} \frac{ x^i}{ | \mathbf x |} \right) \\
		=& \frac{ u''}{1 + Q'^2} 
		+ \frac{ u'}{1 + Q'^2} \sum_i \frac{ \partial }{ \partial x^i} \left( \frac{ x^i}{| \mathbf x |} \right) \\
		&+ \frac{ u'}{ | \mathbf x |} \sum_i x^i \frac{ \partial }{ \partial x^i} \left( \frac{1}{ 1 + Q'^2} \right) 
		+ \frac{ u'}{1 + Q'^2} \frac{ 1}{ |\mathbf x |} \sum_i \frac{ x^i }{ \sqrt{ det g } } \frac{ \partial }{ \partial x^i} \left( \sqrt{ det g } \right) 
	\end{aligned} \end{equation*}

	We compute each of these terms.
	\begin{equation*} \begin{aligned}
		\sum_i \frac{ \partial }{ \partial x^i } \left( \frac{ x^i} {| \mathbf x |} \right)  
		= \sum_i \left( \frac{ 1}{ | \mathbf x |} - \frac{ x^i }{ | \mathbf x |^2} \frac{ x^i }{ |\mathbf x |} \right) 
		= \frac{ n }{| \mathbf x | } - \frac{ | \mathbf x |^2}{ |\mathbf x |^3} 
		= \frac{ n-1}{ | \mathbf x |}
	\end{aligned} \end{equation*}
	
	\begin{equation*} \begin{aligned}
		\sum_i x^i \frac{ \partial }{ \partial x^i} \left( \frac{ 1 }{1 + Q'^2} \right) 
		= \sum_i x^i \left( \frac{ -2 Q' Q''}{ (1 + Q'^2 )^2} \frac{ x^i}{ | \mathbf x |} \right) 
		= \frac{ -2 Q' Q''}{ (1 + Q'^2 )^2} | \mathbf x |
	\end{aligned} \end{equation*}
	
	At a point of the form $\mathbf x  = (x^1, 0, ...., 0)$
	\begin{equation*} \begin{aligned}
		&\sum_i \frac{ x^i}{ \sqrt{ det g}} \frac{ \partial}{ \partial x^i} \sqrt{ det g}	\\
		=& \sum_i \frac{ x^i }{ \sqrt{ det g} } \frac{ 1}{2 \sqrt{ det g} }  ( det g) tr( g^{-1} \partial_{x^i} g ) \\
		=& \sum_i \frac{1}{2} x^i tr ( g^{-1} \partial_{x^i} g ) \\
		=& \sum_i \frac{ 1}{2} x^i
		tr \left[ 
		\left( \begin{array}{ccc} 
			\frac{1}{ 1 + Q'^2} & 0 & 0 \\
			0& Id_{n-1} & 0 \\
			0 & 0 & Q^{-2} g_{\mathbb{S}^{n-1} } \\
		\end{array} \right)
		\frac{ \partial }{ \partial x^i}
		\left( \begin{array}{ccc} 
			1 + Q'^2 & 0 & 0 \\
			0 & Id & 0 \\
			0 & 0 & Q^2 g_{\mathbb{S}} \\
		\end{array} \right)
		\right] \\
		=& \sum_i \frac{ 1}{2} x^i
		tr \left[ 
		\left( \begin{array}{ccc} 
			\frac{1}{ 1 + Q'^2} & 0 & 0 \\
			0& Id_{n-1} & 0 \\
			0 & 0 & Q^{-2} g_{\mathbb{S}^{n-1} } \\
		\end{array} \right)
		\left( \begin{array}{ccc} 
			2 Q' Q'' \frac{ x^i}{ | \mathbf x |} & 0 & 0 \\
			0 & 0 & 0 \\
			0 & 0 & 2 Q Q' \frac{ x^i}{ |\mathbf x |} g_{\mathbb{S}} \\
		\end{array} \right)
		\right] \\
		=& \sum_i \frac{ 1}{2} x^i
		\left[ \frac{ 2 Q' Q'' }{1 + Q'^2} \frac{ x^i}{ | \mathbf x |} + (n-1) \frac{2 Q'}{ Q} \frac{ x^i }{ |\mathbf x |} \right] \\
		=&  \left( \frac{ Q' Q''}{ 1 + Q'^2} + (n-1) \frac{ Q'}{Q} \right) | \mathbf x |
	\end{aligned} \end{equation*}

	It follows that
	\begin{equation*} \begin{aligned}
		\Delta_\Sigma u 
		&= \frac{ u''}{ 1 + Q'^2} 
		+ \frac{ u'}{ 1 + Q'^2} \frac{ n-1}{ | \mathbf x |}
		+ \frac{ -2 Q' Q''}{ ( 1 + Q'^2)^2} u'
		+  \left( \frac{ Q' Q''}{ 1 + Q'^2} + (n-1) \frac{ Q'}{Q} \right) \frac{ u'}{ 1 + Q'^2} \\
		&= \frac{ 1}{1 + Q'^2} \left[ u'' + \frac{ n-1}{ | \mathbf x |} u' - \frac{  Q' Q''}{ 1 + Q'^2} u' +  ( n-1) \frac{ Q'}{ Q} u'\right]
	\end{aligned} \end{equation*}

\end{proof}

\begin{cor} \label{LaplaceBeltramiMinl}
	Let $( \Sigma, g_\Sigma)$ be an $O(n) \times O(n)$-invariant hypersurface
	and $u = u( | \mathbf x | )$ as above.
	If $\Sigma$ is minimal then
		$$\Delta_\Sigma u = \frac{ u''}{ 1 + Q'^2} + \frac{ n-1}{ | \mathbf x |} u'
		= \Delta_{\mathbb{E}^{2n} } u - \nabla^2 u( \nu_\Sigma, \nu_\Sigma) $$
\end{cor}
\begin{proof}	
	For minimal hypersurfaces in $\mathbb{E}^{2n}$, it is a general fact that
		$$\Delta_\Sigma u = \Delta_{\mathbb{E}^{2n} } u - \nabla^2 u ( \nu_\Sigma, \nu_\Sigma)$$
	Thus, it suffices to prove the first equality.

	If $\Sigma$ is minimal, then $Q$ satisfies
	\begin{equation*} \begin{aligned}
		& &	&\frac{ Q''}{ 1 + Q'^2} + \frac{ (n-1)}{| \mathbf x |} Q' - \frac{ n-1}{Q} = 0	\\
		&\implies& &\frac{ n-1}{ | \mathbf x |} Q' = - \frac{ Q''}{1 + Q'^2} + \frac{ n-1}{Q} 	\\
		&\implies& &\frac{ n-1}{ | \mathbf x |} Q'^2 u' = - \frac{ Q'' Q'}{1 + Q'^2} u' + (n-1)\frac{ Q'}{Q}u' 
	\end{aligned} \end{equation*}
	It then follows that
	\begin{equation*} \begin{aligned}
		\Delta_\Sigma u
		&= \frac{ 1}{1 + Q'^2}	\left[ u'' + \frac{ n-1}{ | \mathbf x |} u' - \frac{  Q' Q''}{ 1 + Q'^2} u' +  ( n-1) \frac{ Q'}{ Q} u'\right] \\
		&= \frac{1}{ 1 + Q'^2} \left[ u'' + \frac{ n-1}{ | \mathbf x |} u' + \frac{ n-1}{| \mathbf x |} Q'^2 u' \right] \\
		&= \frac{ u''}{1 + Q'^2} + \frac{ n-1}{ | \mathbf x |} u'
	\end{aligned} \end{equation*}
	This proves the first equality.
\end{proof}

\section{Constants} \label{appendixConstants}

\begin{itemize}

	\item $n \ge 4$ encodes the dimension
	
	\item The parameter $\alpha = \alpha(n)$ depends only dimension and is given by
		$$\alpha = \alpha(n) = \frac{ 1}{2} \left( -(2n-3) + \sqrt{ (2n-3)^2 - 8(n-1)} \right) < 0$$
	$|\alpha(n)|$ is a decreasing function of $n$ for $n \ge 4$.
	Moreover, 
		$$| \alpha(4) | = 2\qquad | \alpha(5) |< \frac{3}{2} \qquad \text{ and }  |\alpha(n)| \searrow 1  \text{ as } n \nearrow \infty$$
		
	\item $\lambda = \lambda_k$ is the eigenvalue given by \comment{changed Velazquez's $l$ to a $k$}
		$$\lambda_k = \frac{\alpha-1}{2} + k		\qquad (k \in \mathbb{N})$$
	We only consider the large enough $k$ for which $\lambda_k > 0$, or, equivalently, we only consider $k \ge 2$.
	
	Moreover, for $n \ge 4$ and $k \ge 2$, $2 \lambda_k -1 \ge 0$.
	Indeed,
	\begin{gather} 
		2 \lambda_k -1
		\ge 2 \lambda_2 -1	
		= 4 - 1 + \alpha -1	
		= 2 - | \alpha | 		
		\ge 0
	\end{gather}
	
	\item $\sigma_k$ encodes the blow-up rate of the second fundamental form
		$$\sigma_k = \frac{\lambda_k}{1 + | \alpha | } > 0$$
\end{itemize}


\bibliography{bddHv2Bib}{}
\bibliographystyle{alpha}
\nocite{*}


\end{document}